\documentclass{article}
\usepackage[margin=3cm]{geometry}
\usepackage{amsmath,amssymb,amsthm}
\usepackage[colorlinks=true,linkcolor=blue,citecolor=blue]{hyperref}
\usepackage{graphicx} 
\usepackage[subpreambles=true]{standalone} 
\usepackage{import} 

\usepackage{tikz}
\usepackage{tikz-cd}
\title{Skein Categories in Non-semisimple Settings}
\author{Jennifer Brown and Benjamin Ha\"ioun}
\date{\today}

\newtheorem{counter}{Counter}
\newtheorem{theorem}[counter]{Theorem}
\newtheorem{lemma}[counter]{Lemma}
\newtheorem{corollary}[counter]{Corollary}
\newtheorem{proposition}[counter]{Proposition}
\theoremstyle{definition}
\newtheorem{definition}[counter]{Definition}
\newtheorem{remark}[counter]{Remark}

\numberwithin{counter}{section}
\numberwithin{equation}{section}

\newcommand{\idty}{{\mathrm{1}\mkern-4mu{\mathchoice{}{}{\mskip-0.5mu}{\mskip-1mu}}\mathrm{l}}}

\newcommand\oT{\overline{T}}
\renewcommand\AA{\mathcal{A}}
\renewcommand\SS{\mathcal{S}}
\newcommand\BB{\mathcal{B}}
\newcommand\CC{\mathcal{C}}
\newcommand\DD{\mathcal{D}}
\newcommand\EE{\mathcal{E}}
\newcommand\II{\mathcal{I}}
\newcommand\JJ{\mathcal{J}}
\newcommand\LL{\mathcal{L}}
\newcommand\VV{\mathcal{V}}
\newcommand\Rr{\mathbb{R}}
\newcommand\inj{\hookrightarrow}
\DeclareMathOperator{\Vect}{Vect}
\DeclareMathOperator{\Sk}{Sk}
\DeclareMathOperator{\SkMod}{SkMod}
\DeclareMathOperator{\SkFun}{\underline{Sk}}
\DeclareMathOperator{\SkAlg}{SkAlg}
\DeclareMathOperator{\SkCat}{SkCat}

\DeclareMathOperator{\Dist}{Dist}
\DeclareMathOperator{\Disk}{\mathbb{D}isk}
\DeclareMathOperator{\Surf}{\mathbb{S}urf}
\DeclareMathOperator{\Bimod}{Bimod}
\DeclareMathOperator{\Cat}{Cat_{\Bbbk}}

\DeclareMathOperator{\colim}{colim}
\DeclareMathOperator{\const}{const}
\DeclareMathOperator{\RT}{RT}
\DeclareMathOperator{\End}{End}
\DeclareMathOperator{\Hom}{Hom}
\DeclareMathOperator{\id}{id}
\DeclareMathOperator{\Nat}{2Nat}

\DeclareMathOperator{\Fun}{Fun}

\begin{document}

\maketitle

\begin{abstract}
We introduce a version of skein categories of surfaces which depends on a tensor ideal in a linear ribbon category, thereby extending the existing theory to the setting of non-semisimple TQFTs.
We obtain modified notions of skein algebras of surfaces and skein modules of 3-cobordisms for non-semisimple ribbon categories.

We prove that these skein categories built from ideals coincide with factorization homology, shedding new light on the similarities and differences between the semisimple and non-semisimple settings.
The essential difference is the need to work with profunctors in the non-semisimple setting. Doing so produces a ``distinguished presheaf'' which plays the role of the distinguished object in skein categories in semisimple settings. 
As a consequence, we get a skein-theoretic description of factorization homology for a large class of balanced braided presentable categories, precisely all those which are expected to induce oriented categorified 3-TQFTs.
\end{abstract}
\tableofcontents

\section{Introduction}
We pass between two main perspectives in this work.
First is a topological viewpoint, emphasizing the role of skeins, while the second is algebraic, emphasizing that of factorization homology.
These two approaches complement each other, since skeins give us topological intuition for the abstract constructions provided by factorization homology, which in turns gives us a toolkit for understanding the explicit skein constructions.
The general plan of this paper is to establish then take advantage of this exchange to study non-semisimple skein theory.

\subsection{Background and motivation}
Skein theory was developed in the 1990s \cite{PrzytyckiSkeinModules,TuraevSkeinModules, RT90}, as a sweeping generalization of Vaughan Jones' construction of knot invariants  \cite{Jones_1985}.
It is the basis for a family of link and 3-manifold invariants known collectively as Reshetikhin--Turaev invariants and has close ties to Topological Quantum Field Theories (TQFTs) in dimension 3 \cite{TuraevBook, BHMV95} and in dimension 4 \cite{CYK, WalkerNotes}.
The construction is based on a graphical calculus built using decorated oriented framed graphs (``\emph{ribbon graphs}") embedded in 3-dimensional space, whose decorations and relations depend on the choice of a ribbon category.
To better suit future work on TQFTs and simplify our presentation, we will only consider $\Bbbk$-linear ribbon categories over a field $\Bbbk$.

A central notion is that of a \emph{skein module} spanned by ribbon graphs embedded in a fixed 3-manifold, the best known of which is the Kauffman bracket skein module.
These have been studied extensively, see \cite{Garoufalidis_Yu_2024,Panitch_Park_2024,GJS, KinnearDimSkein, DKSnontrivKBSM} for a few recent developments and \cite{Przytycki_Bakshi_Ibarra_Montoya-Vega_Weeks_2024} for a historical overview.
The skein module of a thickened surface inherits an algebra structure from stacking and is therefore referred to as the \emph{skein algebra}. It has been extensively studied, particularly in the case of $\mathrm{SL}_2$ where it a quantization of the $\mathrm{SL}_2$-character variety of the surface, \cite{BullockKBSM, PrzytyckiSikoraKBSA}.
When the deformation parameter $q$ is a root of unity, the ribbon category $\operatorname{Rep}_q(SL_2)$ is non-semisimple, and the associated skein algebra is an object of great current interest \cite{BonahonWongRepKBSA, BonahonWongRepKBSAIIIUnicityConj, FKLunicityKBSA, GJSunicity}.

The \emph{skein category} of a surface \cite{WalkerNotes, JohnsonFreydHeisenbergPicture} is a $\Bbbk$-linear category whose hom-spaces are spanned by the same colored ribbon graphs which define skein algebras and modules.
We focus on skein categories (Definition \ref{def:modified-skein-cat}) in part because they underpin the construction of skein algebras and modules (Sections \ref{sec:mod-skein-algebras} and \ref{sec:mod-skein-modules}), and in part because they compute factorization homology (Theorem \ref{thm:main-result}). 

\paragraph{Non-semisimple Skein Theory}
Non-semisimple ribbon categories appear naturally in physical applications, for example in Chern-Simons theories coupled with matter sectors \cite{CreutzigDimofteGarnerGeerNonssQFT}, or in Chern-Simons theories built from supergroups \cite{Rozansky_Saleur_1992,Mikhaylov_2015_thesis}, which have recently gained attention in skein theory \cite{Garner_Geer_Young_2025,Geer_Young_2025}.
The usual construction of a ``Kirby color", used to define invariants of 3-manifolds from link invariants in the semisimple case \cite{RT90}, no longer applies in non-semisimple theories.
This problem was first overcome by Hennings \cite{Hennings3mfldInvHopf} in the Hopf-algebraic case and by Lyubashenko \cite{Lyub3mfldProjMCG} in a more general setting.
These developments led to partially defined non-semisimple TQFTs for connected manifolds with corners \cite{KerlerLyubBook}.

Additionally, when the Reshetikhin--Turaev construction is applied to a non-semisimple ribbon category, the associated link invariant is often identically zero.
This was overcome by the theory of \emph{modified traces} \cite{GeerPatureauTuraevModifiedqdim, GKPgeneralizedTraceModdim,CGP3manInv}, which have fueled the rigorous mathematical construction of non-semisimple 3-manifold invariants and TQFTs \cite{CGP3manInv, BCGP, DGGPR, CGHP}.
Modified traces are defined relative to a tensor ideal $\II$ in the ambient ribbon category.
The smallest non-trivial tensor ideal in any tensor category is that formed by projective objects.
If the tensor unit is projective, as is the case for semisimple categories, then the category has no non-trivial proper tensor ideals and the modified trace is exactly the quantum trace used in the original Reshetikhin--Turaev constructions.

Another key ingredient in non-semisimple skein-theoretic constructions is the notion of \emph{admissible skein modules} \cite{CGPAdmissbleskein}, which likewise depend on the choice of a tensor ideal $\II$. 
We refer to this family of skein-theoretic constructions well adapted to non-semisimple ribbon categories as \emph{modified skein theory} or $\II$-skein theory.
The closely related theory of string nets \cite{LevinWenStringNets} followed a similar path \cite{MSWYnonssStringNet, CGPVnc2plus1} towards non-semisimple generalization. 

Work to translate unmodified objects and properties into the modified setting is significant and ongoing.
Many standard constructions and results in the semisimple case, for example comparisons between Reshetikhin--Turaev and Crane--Yetter or Turaev--Viro-type TQFTs are yet to appear in the non-semisimple case. 

A persistent theme in this larger project is working around the fact that the monoidal unit of a non-semisimple ribbon category is not projective and therefore typically forbidden as a labeling object.
Indeed, the tensor ideal $\II$ on which modified traces and admissibility conditions are defined typically will not contain the unit, in which case the admissible skein modules will not contain the empty skein.
In the current paper this manifests as a need to work with presheaf-valued functors, see Remark \ref{rmk:fake-empty}.

\paragraph{Factorization Homology}
Factorization homology was introduced by Ayala and Francis \cite{Ayala_Francis_2019}. 
It gives a way of ``integrating" a $\Disk_n$-algebra\footnote{This is equivalent to the notion of an $\mathbb{E}_n$-algebra, see \cite[Remark 2.28]{Ayala_Francis_2019}.} $A$ in a symmetric monoidal $\infty$-category $\SS$, called the coefficients, over an $n$-manifold $M$.
Doing this, it produces an object $\int_MA\in\SS$ which is functorial in both $M$ and $A$.
We will take $n=2$, take $\SS$ to be a $(2,1)$-category, fix $A$, and contemplate $\int_-A$ as a symmetric monoidal functor on the (2,1)-category $\Surf$ of surfaces, embeddings and isotopies.

We will contemplate three choices of (2,1)-category $\SS$. 
The first is the (2,1)-category $\Cat$ of $\Bbbk$-linear categories, functors and natural transformations. 
The second $\Bimod$ (recalled in Definition \ref{def:Bimod}) has the same objects and 2-morphisms, but allows functors to be presheaf-valued, i.e. a morphism from $\CC$ to $\DD$ is a functor $\CC\to \Fun(\DD^{op},\Vect_\Bbbk)$. 
The third is $\Pr$ the category presentable cocomplete categories and cocontinuous functors. 
It contains $\Bimod$ as the full subcategory of presheaf categories. 
Free cocompletion gives a 2-functor $\Cat \inj \Pr$ that sends a category to its presheaf category.

The ribbon categories used to define skein theory are in one-to-one correspondence with a collection of well behaved $\Disk_2$-algebras in $\Cat$, which in general correspond to balanced braided categories \cite[Example 5.1.2.4]{LurieHA} \cite[Chapter 6]{FresseE2Operads}.
By the main result of \cite{CookeExcision}, the skein category of a small $\Bbbk$-linear ribbon category computes factorization homology of the corresponding $\Disk_2$-algebra in $\Cat$, see also \cite{KirillovThamFH4DTQFT} for a similar result and \cite{BrochierWoike, MullerWoikeAdmissibleskein} for other comparisons with factorization homology.

Passing to the free cocompletion gives a skein theoretic description of $\int_{\Sigma}\EE$ for those $\Disk_2$-algebras $\EE$ in $\Pr$ which are the free cocompletion of a $\Disk_2$-algebra in $\Cat$.
Note that these constructions require that the tensor unit is projective, which is an imperfect indication of semisimplicity \cite[Corollary 4.2.13]{EGNO}.

In Corollary \ref{cor:FHofE2inPr} we extend the skein-theoretic description of factorization homology to a larger family of $\Disk_2$-algebras in $\Pr$. 
Namely, we obtain a description for any disk algebra in $\Pr$ whose assignment to a disk is a \emph{cp-ribbon} category, i.e. is cocomplete, \emph{cp-rigid} (has enough compact-projective objects, all of which are dualizable), and whose balancing induces a ribbon structure on the subcategory of dualizable objects.

This is a generalization of \cite{CookeExcision} because the free cocompletion of a ribbon category in $\Cat$ meets these conditions, but not every cp-ribbon $\Disk_2$-algebra in $\Pr$ comes from such a free cocompletion.

It turns out that cp-ribbon captures exactly the conjectural conditions for a $\Disk_2$-algebra in $\Pr$ to define a fully extended categorified 3-TQFT. 

\paragraph{Topological Quantum Field Theories}

Skein categories were introduced in \cite{WalkerNotes} to describe Crane--Yetter 4-TQFTs on surfaces (see \cite{JohnsonFreydHeisenbergPicture, CookeExcision,GJS} for more formal treatments), but not every choice of ribbon category is expected to correspond to a fully extended four-dimensional skein theory.

Thanks to their agreement with factorization homology and the work of \cite{Scheimbauer}, it is known that skein categories are part of a twice-categorified fully-extended 2-TQFT.
In Theorem \ref{thm:its-a-TQFT} we describe the three-dimensional part of this TQFT in a non-semisimple context, but do not discuss its extension to dimensions zero, one, or four.

In the current work we are interested in those disk algebras that are expected (see \cite[Conjecture 9.10]{JohnsonFreydHeisenbergPicture}) to define a fully extended categorified three-dimensional theory. 
According to the cobordism hypothesis \cite{Baez_Dolan_1995,LurieCob}, these are exactly those $\Disk_2$-algebras in $\Pr$ which are 3-dualizable and 3-oriented in an appropriate $(\infty,4)$-category (see \cite{JFS} for its formal definition.)
The family of disk algebras whose factorization homology we describe skein-theoretically is delineated by a combination of the definition of \emph{cp-rigid} that characterizes 3-dualizability in \cite{BJS,BJSS}, and the existence of the ribbon structure expected to be equivalent to 3-oriented (this is still an open question, but see \cite{SchommerPries14, Scheimbauer}.)

In future work we will be interested in those categorified oriented 3-TQFTs which extend to oriented 4-TQFTs.
These should correspond via the cobordism hypothesis to 4-dualizable 4-oriented objects. 
Walker's original skein-theoretic description \cite{WalkerNotes} of Crane--Yetter theories extends to 4-manifolds under the assumption that the input ribbon category is fusion. 
Being fusion and braided is a sufficient \cite[\S 5.6]{BJS} but not necessary condition for being 4-dualizable.
It is shown in \cite{BJSS} that Ind-completions of non-semisimple modular tensor categories are 4-dualizable, even though they are not fusion.
They have non-compact-projective unit and therefore cannot be obtained as the free cocompletion of a ribbon category\footnote{In fact, it follows from \cite{HaiounUnit, StewartPhD} that 4-dualizability of the free cocompletion of a $\Disk_2$-algebra in $\Cat$ \emph{imposes} semisimplicity.}.
The three and four dimensional part of the associated non-semisimple Crane--Yetter theories are constructed using modified skein theory and admissible skein modules \cite{DGGPR, CGHP}.

Giving a skein theoretic description to the two dimensional part of this theory was one of the main motivations for the current article, see Corollary \ref{cor:FHofE2inPr}.
We caution however that simply having a description of what a fully extended theory assigns to spaces of every dimension is far from a complete description. The missing information involves the many coherence and gluing relations demanded of a TQFT.

\subsection{Goals}
Here we state the main questions which motivated and organize this paper.
\newtheorem{question}{Question}
\newtheorem{subquestion}{Question}
\renewcommand{\thesubquestion}{\thequestion\alph{subquestion}}
\setcounter{question}{1}

\begin{subquestion}\label{question:SkCat}
    What is the appropriate adaptation of skein categories assigned to surfaces to non-semisimple ribbon categories?
\end{subquestion}
We argue that modified skein categories with coefficients in a tensor ideal $\II$ of a braided tensor category $\AA$ (Definition \ref{def:modified-skein-cat}) are the natural adaptation, given their relation with factorization homology (Theorem \ref{thm:main-result}).
An algebraic formulation of Question \ref{question:SkCat} could be stated as follows:
\begin{subquestion}\label{question:FH}
Is there a skein theoretic description of the factorization homology of a cp-ribbon category whose unit is not compact-projective? 
\end{subquestion}
From this standpoint, Theorem \ref{thm:main-result} (more precisely its reformulation in Corollary \ref{cor:FHofE2inPr})
states that 
 $\II$-skein categories give more concrete descriptions of an existing theory: the factorization homology of non-semisimple categories.
Yet another way of phasing this question is: 
\begin{subquestion}\label{question:4TQFT}
    What does the fully-extended non-semisimple Crane--Yetter theory assign to surfaces?
\end{subquestion}
\setcounter{subquestion}{0}

Even though such a fully extended theory is not rigorously defined yet, see \cite{WalkerNotes} for a sketch in the semisimple case, our results suggest that skein categories of tensor ideals model their values on surfaces.

\begin{question}\label{question:its-FH}
Can we give a version of the proof \cite[Theorem 2.28]{CookeExcision} which does not appeal to excision and which holds in non-semisimple settings?
\end{question}

This is accomplished by the proof of Theorem \ref{thm:main-result}, which drives many constructions in this paper.

\subsection{Results}
In Definition \ref{def:modified-skein-cat} we introduce the skein category associated to a tensor ideal $\II$ in a ribbon category $\AA$ which we call the $\II$-skein category\footnote{We remove $\AA$ from the notation as our notion ultimately only depends on $\II$. The presence of $\AA$ is a nice way to express a list of conditions $\II$, see Remark \ref{rmk:dep_on_II}.
}.
The motivating example of such an ideal is the subcategory of projective objects in a non-semisimple ribbon tensor category, in which case we propose our construction is the one sought in Question \ref{question:SkCat}. 

Our main technical result is that $\II$-colored skein categories compute factorization homology:
\newtheorem*{thm:main-result}{Theorem \ref{thm:main-result}}
\begin{thm:main-result}\it
    There is an equivalence of symmetric monoidal 2-functors from $\Surf$ to $\Bimod$ $$\displaystyle \int_-\II \simeq \SkCat_\II(-)$$ between factorization homology with coefficients in $\II$ seen as an $\Disk_2$-algebra in $\Bimod$ and the functor associated to the $\II$-skein category construction.
\end{thm:main-result}
This suggests that $\SkCat_\II$ is the canonical adaptation of skein categories to the non-semisimple setting, in the sense that it satisfies the same universal properties. 
Our definition closely follows what is assigned to surfaces by the TQFTs defined in \cite{BCGP, DGGPR}, and in \cite{MSWYnonssStringNet} in one dimension lower. 

We emphasize that the version of skein categories introduced in this work gives a 2-functor from surfaces and embeddings to $\Bimod$, not $\Cat$. 
This seemingly minor detail powers many of our constructions and proofs, see Remark \ref{rmk:fake-empty}. Functoriality in $\II$ is discussed in Remark \ref{rmk:functoriality_in_II}.

Reformulating our main result, we answer Question \ref{question:FH}:
\newtheorem*{cor:FHofE2inPr}{Corollary \ref{cor:FHofE2inPr}}
\begin{cor:FHofE2inPr}{\it
    Let $\EE$ be a \emph{cp-ribbon} category in $\Pr$, i.e. $\EE$ is balanced, braided, cp-rigid, and its balancing is a ribbon structure on its subcategory of dualizable objects. 
Let $\II:=\EE^{cp}$ denote the subcategory of compact-projective objects. 
We have an equivalence of symmetric monoidal 2-functors $$\displaystyle \int_-\EE \simeq \widehat{\SkCat}_{\EE^{cp}}(-)$$
     between factorization homology in $\Pr$ with coefficients in $\EE$ and free cocompletions of $\II$-skein categories.
   }
\end{cor:FHofE2inPr}

Corollary \ref{cor:FHofE2inPr} echos a stronger prediction by David Jordan: that the 3-TQFT associated to these cp-ribbon categories is described by admissible skein modules. 

By analogy with the traditional construction in both skein theory and factorization homology, we define the modified skein algebra with coefficients in $\II$ of a surface $\Sigma$ to be the endomorphism algebra of the distinguished presheaf $\Dist_\Sigma \in \widehat{\SkCat}_\II(\Sigma)$ induced by the inclusion of the empty surface $\emptyset \inj \Sigma$: $$\SkAlg_\II(\Sigma) := \End_{\widehat{\SkCat}_\II(\Sigma)}(\Dist_\Sigma)\ .$$
We show (Remark \ref{rmk:mod-SkAlg-acts-on-adm-SkMod}) that this algebra acts 
on admissible skein modules. The action of the usual skein algebra on admissible skein modules factors through this action.
This construction is natural with respect to embeddings of surfaces, so a surface's mapping class group acts on its $\II$-skein algebra by algebra homomorphisms.

For connected surfaces with non-empty boundary we can use the algebraic tools developed for factorization homology in \cite{BBJ} to get an algebraic formulation of $\II$-skein algebras as follows. 
First, the $\II$-skein algebra is an invariant subalgebra of the internal endomorphism algebra \cite[Def. 3.9 (5)]{BBJ} of the distinguished object $\underline{\End}(\Dist_\Sigma)$, see also \cite[Def. 5.3]{BBJ}.
They show that the moduli algebra is isomorphic as an algebra object in $\widehat{\mathcal{I}}$ to a tensor power of Lyubashenko's coend $\LL$ \cite{LyubashenkoModtrasnfoTensorCats, Lyub3mfldProjMCG}, leading to the formula for our $\II$-skein algebras in Corollary \ref{cor:ModSkAlgIsInvCoend}.
Using the results of \cite{BBJ2}, we obtain a similar description for surfaces with empty boundary via quantum Hamiltonian reduction.

We also construct $\II$-skein modules of 3-cobordisms. 
Surprisingly, an in contrast to the semisimple case, these modules depend on the boundary decomposition that distinguishes a cobordism from its underlying 3-manifold.
This is because the construction treats incoming and outgoing boundaries differently.
Playing with decompositions of the boundary, we can either recover admissible skein modules or a new notion which allows the empty ribbon graph, see Remark \ref{rmk:mod-SkAlg-acts-on-adm-SkMod}.

\paragraph{Modified Traces}
Modified traces are an iconic features of non-semisimple skein theory.
They are in one-to-one correspondence with linear forms on the admissible skein module of the 3-sphere $S^3$ \cite[Cor. 3.2]{CGPAdmissbleskein} and appear most famously when constructing link invariants in this setting \cite{GeerPatureauTuraevModifiedqdim}.
They are not, however, a major feature in this paper.
The work here reinforces the idea that modified traces are a natural notion, as opposed to a useful but ad-hoc construction.

The precise relationship is as follows.
We show that modified skein categories are the canonical extension (i.e. factorization homology) of non-semisimple braided presentable categories to surfaces.
This is the 2-dimensional part of a TQFT whose 3-dimensional part is described by admissible skein modules (conjecturally, this is again the canonical extension, given by the cobordism hypothesis).
One could extend this skein theoretic TQFT to 4-manifolds via handle attachments.
In order to do so one needs a value for the 4-handle $\mathbb{D}^4:S^3 \to \emptyset$.
This value is given by a modified trace.

\paragraph{Terminology}
Our constructions are close enough to the existing ones\footnote{In particular, when $\II = \AA$ we recover existing constructions of skein algebras, modules and categories.} that we are reluctant to introduce new nomenclature beyond decorating by $\II$ instead of $\AA$.
Still, it's sometimes useful to have a tentative name to distinguish new from old.
A few alternatives for the name ``modified skeins'' were seriously considered.
Admissible or non-unital skeins seemed to be equally good candidates, since the skein modules of 3-manifolds are called admissible in \cite{CGPAdmissbleskein} and the tensor ideal $\II$ is a non-unital ribbon category.
In the end we settled on `modified' because, perhaps surprisingly, both `non-unital' and `admissible' would have been misleading names for the associated skein algebras. 
For one, these algebras are in fact unital.
They also don't have an apparent admissibility condition.
\emph{Any} $\AA$-colored skein in $\Sigma\times [0,1]$ gives an element of the modified skein algebra -- including the empty skein and those without $\II$-colored strands. 

In fact, despite the initial restriction to a tensor ideal, the modified skein algebra in general has \emph{more elements} than the traditionally constructed version, see Section \ref{sec:less-is-more}. Notably they include Lyubashenko's non-semisimple replacement for the Kirby color, which he used to interpret 3-manifold surgery \cite[p.8 and Sec. 7.2]{Lyub3mfldProjMCG}.

Given this, we reserve the term `admissible' for objects such as labellings, ribbon graphs, relative skein modules, and skein relations where there is indeed an evident admissibility condition.
We avoid it for constructions such as the skein algebra or category associated to a surface where these initial restrictions do not manifest as straightforward admissibility constraints.

\subsection{Future directions}

\paragraph{Connections to TQFTs}
We do not address Question \ref{question:4TQFT} above, one difficulty being that non-semisimple (or even semisimple) Crane--Yetter theories are not yet formally defined as fully extended TQFTs. However, there are two strong arguments for why this theory should assign $\II$-skein categories to surfaces. First, by the uniqueness part of the cobordism hypothesis, this theory would have to agree on surfaces with factorization homology, hence with $\II$-skein categories by the main result of this paper. Second, it was shown in \cite{HaiounWRTdelCY} (using the results of the current paper, see also \cite{WalkerNotes} for the semisimple case) that $\II$-skein categories extend the (3+1)-TQFTs of \cite{CGHP} down to surfaces. There is also announced ongoing work of Reutter--Walker in this direction. 

Using this extension to surfaces, the second author proved in \cite{HaiounWRTdelCY} that Witten--Reshetikhin--Turaev and \cite{DGGPR} 3-TQFTs live at the boundary of Crane--Yetter 4-TQFTs. This is an expected result proposed by Walker, Freed and Teleman \cite{WalkerNotes}. In particular, our study provides some insight into the construction of \cite{DGGPR} TQFTs.

For example, the (2+1)-TQFTs of \cite{DGGPR} have been extended to the circle in \cite{DeRenziDGGPR}. 
However, this once-extended description imposes admissibility conditions on surfaces and surprisingly does not recover all of \cite{DGGPR}. 
We claim that this obstruction comes from the fact that the extension was given with values in $\Cat$, whereas some maps between $\II$-skein categories only exist in $\Bimod$. 
Indeed, according to the boundary condition description of \cite[Sec. 6]{HaiounWRTdelCY}, the theory \cite{DGGPR} on a closed surface $\Sigma$ is obtained by composing the boundary condition $\Vect \to \SkCat_\II(\Sigma)$ induced by the embedding $\emptyset \inj \Sigma$ with the skein module functor on a bounding 3-manifold. However, as we note in Remark \ref{rmk:fake-empty}, the boundary condition $\Vect \to \SkCat_\II(\Sigma)$ is not a functor (a morphism in $\Cat$), but only a pro-functor (a morphism in $\Bimod$).

\paragraph{New elements and relations}
In Section \ref{sec:less-is-more} we describe new elements of $\SkAlg_\II(\Sigma)$, which do not come from elements of the usual skein algebra $\SkAlg_\AA(\Sigma)$, in terms of bichrome graphs.
Using the results of \cite{BBJ, BBJ2} it should be possible to give a much finer description. 
In the special case of categories of modules over a finite-dimensional ribbon Hopf algebra, the moduli algebras and their algebras of invariants are studied in \cite{BFRqmoduliIII} and in some examples an explicit basis is known, see \cite{FaitgSLFsl2}.
We hope to make similarly concrete descriptions in future works, and in particular to describe the image, kernel, and cokernel of the map \ref{def:AA-to-II-skalg-map}.

In concrete examples we can show that the claimed ``new" elements coming from bichrome graphs cannot be described by usual skeins.
For example, let $\AA$ be the category of representations of the small quantum group for $SL_2$ at a $2p\, $th root of unity, $\II$ the ideal of projective objects, and $\Sigma$ an annulus.
In this case, the $\II$-skein algebra is $3p{-}1$-dimensional while the image of the map $\SkAlg_\AA(\Sigma) \to \SkAlg_\II(\Sigma)$ is $2p$-dimensional.
See \cite{FaitgSLFsl2} and references therein for explicit computations, noting that they call elements of $\SkAlg_\II(\Sigma)$ symmetric linear forms and call elements in the image of $\SkAlg_\AA(\Sigma) \to \SkAlg_\II(\Sigma)$ characters. To make the connection precise one uses that the algebra of symmetric linear forms and the modified skein algebra of the annulus are both isomorphic to the subalgebra of invariants in the canonical coend.

\paragraph{Stated skeins}
It is natural to ask for a non-semisimple generalization of the main result of \cite{HaiounStated}, which shows that \cite{BBJ}'s moduli algebras are isomorphic to \cite{LeStated}'s stated skein algebras for the ribbon category of $\mathcal{U}_q(\mathfrak{sl}_2)$-modules at generic $q$. 
After discussing with Francesco Costantino and Matthieu Faitg, we have come to believe that the notion of \emph{stated $\II$-skein algebras} should coincide with stated $\AA$-skein algebras when there are marked points in every connected component and $\II= \operatorname{Proj}(\AA)$. 
The marked points in some sense flatten the theory, quelling differences between the modified/usual construction. 
There might be a derived explanation for this phenomenon, linked with the fact that representations varieties are smooth whereas character varieties can be singular.

\subsection*{Acknowledgements} We are grateful to Francesco Costantino, Matthieu Faitg, Theo Johnson-Freyd, David Jordan, and Eilind Karlsson for enlightening conversations which informed and motivated this work. 
We thank Lukas M\"uller, Lukas Woike, and Adrien Brochier for kindly sharing ideas for a proof of our main result. 
Although we ended up taking an alternative approach, such discussions were a source of motivation and inspiration.

Earnest discussion of this project began while both authors were at a meeting of the Simons Collaboration on Global Categorical Symmetry at the SwissMap Research Station.
We are grateful to the funders and organizers of that event for providing us with a productive and agreeable atmosphere in which to pursue new work.

JB is funded by the Simons Foundation award 888988 as part of the Simons Collaboration on Global Categorical Symmetry, and partially supported by the National Science Foundation under Award No. 2202753.

\paragraph{Note}  Shortly after the prepublication of this paper, we learned about independent work of Runkel, Schweigert and Tham \cite{RunkelSchweigertThamExciAdmSkeins} obtaining similar results, in particular excision properties of admissible skein modules. 
We encourage the interested reader to take a look at their paper for further discussion on this subject.

\section{Topological constructions}
We introduce $\II$-colored skein categories and algebras of surfaces, and $\II$-colored skein modules of 3-manifolds. 
They take as algebraic input a \emph{tensor ideal} $\II$ in a small $\Bbbk$-linear ribbon category $\AA$, i.e. a full subcategory stable under tensoring with any object of $\AA$ and under taking retracts. This implies that $\II$ is stable under taking duals.

We do not need $\AA$ to be abelian or to be tensor in the sense of \cite{EGNO}, and eventually find that our $\II$-skein categories do not depend on $\AA$ itself, only $\II$. (See Remark \ref{rmk:dep_on_II}.)
Tensor ideals appear frequently in the non-semisimple setting, as a defining input to both modified traces \cite{GKPgeneralizedTraceModdim} and admissible skein modules \cite{CGPAdmissbleskein}. 
As mentioned in the introduction, it is typical to consider the ideal of projective objects $\operatorname{Proj}\AA \subseteq \AA$.
This is in part because $\operatorname{Proj}\AA$ is contained in every non-zero ideal whenever the monoidal unit is simple (see e.g. \cite[Lemma 17(c)]{Geer_Patureau-Mirand_Virelizier_2012}.)
Since by definition every object in a semisimple category is projective, this means that semisimple ribbon categories with simple unit have no non-zero proper tensor ideals.

\subsection{Skein categories}\label{sec:mod-skein-cats}
Skein categories were informally introduced in \cite{WalkerNotes} and formally defined in \cite[Def. 9.3]{JohnsonFreydHeisenbergPicture}.
We also use the descriptions from \cite{CookeExcision} and \cite{GJS}.
To avoid confusion, we will not reintroduce the usual notion of skein categories and whenever we say skein categories we mean the modified version of Definition \ref{def:modified-skein-cat}.

Often the objects in skein categories are finite collections of colored oriented framed points in the surface.
To help with the proof of Theorem \ref{thm:main-result}, we give an equivalent definition in terms of embeddings of disks based on the $\AA$-labelings of \cite[Def. 2.1]{GJS}. 
An embedding of the standard disk specifies a framed oriented point as the image of its center.
Its framing is the positive $x$-axis and it is oriented by the orientation of the embedding. 
In the other direction, a collection of framed oriented points in $\Sigma$ corresponds to a contractible space of possible embedded disks. 
\begin{definition}
Let $\Sigma$ be a compact oriented surface, possibly with boundary, $M$ a compact oriented 3-manifold containing $\Sigma$ in its boundary and $\II \subseteq \AA$ a tensor ideal in a ribbon category. 
We also assume that each connected component of $\Sigma$ has a choice of either inward or outward normal.

An \textbf{$\II$-labeling in $\Sigma$} is a pair $(\iota, W)$ where $\iota: d\hookrightarrow \Sigma$ is a (possibly orientation-reversing) embedding of a collection of $\vert d\vert$ standard disks and $W$ is an object of $\II^{\vert d\vert}$.
Let $\vec{X}_\iota$ denote the collection of framed oriented points in $\Sigma$ determined by $\iota$ and colored by the corresponding components of $W$.

An \textbf{$\II$-colored ribbon graph} in $M$ is the image of an embedding  $\Gamma \hookrightarrow M$ of a finite oriented graph $\Gamma$ equipped with a smooth framing.
We require that $\Gamma$ intersects $\partial M$ in $\Sigma$, transversely and only at 1-valent vertices which we call \emph{boundary vertices}.
Each edge of $\Gamma$ is colored by an object of $\II$ and each non-boundary vertex by a morphism as detailed below.
Boundary vertices inherit a color, orientation, and framing from their unique incident edge (we declare that the orientation is positive if the strand is going in the direction of the chosen normal and negative otherwise).
At a non-boundary vertex, the embedding together with the framing induces a cyclic ordering on the incident edges, and we color by an element of $\Hom_\AA(\idty_\AA, V_1^\pm\otimes\cdots\otimes V_n^\pm)$ where $V_1,\dots,V_n$ are the ordered\footnote{Note that the pivotal structure induces a canonical isomorphism  $\Hom_\AA(\idty_\AA, V_1^\pm\otimes\cdots\otimes V_n^\pm)\simeq 
\Hom_\AA(\idty_\AA, V_2^\pm\otimes\cdots\otimes V_n^\pm\otimes V_1^\pm)$, allowing us to upgrade the cyclic ordering on the $V_i$ to a linear ordering.} colors of the incident edges, taken as $V^+=V$ if the edge is outgoing, and $V^-=V^*$ if the edge is incoming.

A ribbon graph is said to be \textbf{compatible} with an $\II$-labeling $(\iota,W)$ if the boundary vertices of $\Gamma$ match the colored oriented framed points $\vec{X}_{\iota}$. See Figure \ref{fig:II_labelings} for an example.
\end{definition}
\begin{figure}
    \centering
\begin{tikzpicture}
    \draw (0,0) arc (0:360:1 and 0.5);
    \fill[gray!10] (0,0) arc (0:360:1 and 0.5);
    \draw (-0.3,0) arc (0:360:0.2 and 0.15);
    \fill[gray!30] (-0.3,0) arc (0:360:0.2 and 0.15);
    \draw (-1.2,0) arc (0:360:0.3 and 0.2);
    \fill[gray!30] (-1.2,0) arc (0:360:0.3 and 0.2);
    \draw (0,-1.5) arc (0:360:1 and 0.5);
    \fill[gray!10] (0,-1.5) arc (0:360:1 and 0.5);
    \node at (-0.5,0) {\small $Y$};
    \node at (-1.5,0) {$X$};
    \begin{scope}[xshift = 4cm]
    \draw (0,-1.5) arc (0:360:1 and 0.5);
    \draw (0,0) rectangle (-2,-1.5);
    \fill[gray!10] (0,-1.5) arc (0:360:1 and 0.5);
    \fill[gray!10] (0,0) rectangle (-2,-1.5);
    \draw (0,0) arc (0:360:1 and 0.5);
    \fill[gray!10] (0,0) arc (0:360:1 and 0.5);
    \draw (-0.3,0) arc (0:360:0.2 and 0.15);
    \fill[gray!30] (-0.3,0) arc (0:360:0.2 and 0.15);
    \draw (-1.2,0) arc (0:360:0.3 and 0.2);
    \fill[gray!30] (-1.2,0) arc (0:360:0.3 and 0.2);
    \node at (-0.5,0) {\small $Y$};
    \node at (-1.5,0) {$X$};
    \draw[thick] (-0.5,0) .. controls (-0.8,-1.2).. (-1,-1.2) node[near start, sloped]{$>$} node{$\bullet$} node[below]{$\scriptstyle f:\idty \to X\otimes Y$};
    \draw[thick] (-1.5,0) .. controls (-1.2,-1.2).. (-1,-1.2) node[near start, sloped]{$<$};
    \end{scope}
\end{tikzpicture}
        \caption{Left: an $\II$-labeling on the disjoint union of two disks. It is not admissible because the lower disk is unlabeled. Right: a compatible $\II$-ribbon graph in the thickened disk. The chosen normal on the two disks is going up on the picture. It is admissible.}
    \label{fig:II_labelings}
\end{figure}

A central result of \cite[Thm 5.1]{RT90}, \cite[Theorem 2.5]{TuraevBook} is the definition of the Reshetikhin-Turaev evaluation $\RT_\AA$ that sends an $\AA$-colored (and therefore $\II$-colored) ribbon graph in the thickened standard disk $\mathbb D \times [0,1]$ compatible with two $\II$-labelings $(\iota,W)$ and $(\iota',W')$ in $\mathbb D \times \{0,1\}$ to a morphism in $\AA$ from the tensor product of the $W$s to the tensor product of the $W'$s.
Note that if $d = \emptyset$ or $d' = \emptyset$ this tensor product may not be in $\II$.
To avoid this and similarly troublesome situations, we introduce the following admissibility condition.

\begin{definition}
Let $\II \subseteq \AA$ be a tensor ideal in a ribbon category and $\Sigma$ a compact oriented surface, possibly with boundary. 
An $\II$-labeling $(\iota,W)$ in $\Sigma$ is called \textbf{admissible} if $\iota: d\hookrightarrow \Sigma$ is surjective on connected components, i.e. $\pi_0(\iota):\pi_0(d)\to\pi_0(\Sigma)$ is surjective.\\
Similarly an $\II$-colored ribbon graph $\Gamma$ is called \textbf{admissible} if $\Gamma \inj M$ is surjective on connected components. See Figure \ref{fig:II_labelings} for an example.
\end{definition}
We now have all the ingredients to define admissible skein modules.
\begin{definition}\label{def:rel-skein-mod} Let $M$ be a compact oriented 3-manifold with $\Sigma\subseteq \partial M$ as above. Given an $\II$-labelings $(\iota,W)$ in $\Sigma$, the \textbf{relative admissible skein module} $\Sk_\II(M;(\iota,W))$ is the vector space freely generated by {isotopy classes} of admissible $\II$-colored ribbon graphs in $M$ compatible with $(\iota,W)$ quotiented by the following \textbf{admissible skein relation}:

Consider a finite family of $\II$-colored ribbon graphs $(T_i)_{i\in I}$ and an oriented embedded standard cube $\varphi:[0,1]^3\inj M$ such that the $T_i$'s coincide strictly outside this cube, intersect its boundary transversely, and only intersect the top and bottom faces. 
We ask moreover that this intersection is non-empty. 
Each $T_i\cap \mathrm{im}\ \varphi$ is sent by the Reshetikhin--Turaev functor to a morphism $\RT_\AA(T_i\cap \mathrm{im}\ \varphi)$ in a common $\Hom$-space in $\AA$.
    
We say $$\sum_{i\in I} \lambda_i T_i \sim 0 \quad\text{ if }\quad \sum_{i\in I} \lambda_i \RT_\AA(T_i\cap \mathrm{im}\ \varphi)=0.$$
\end{definition}

We will often take $M = \Sigma \times [0,1]$ for a compact surface $\Sigma$, possibly with boundary, with $\Sigma\times \{0,1\}\subseteq \partial M$. 
As orientation data we choose the inward normal for $\Sigma\times \{0\}$ and outward normal for $\Sigma\times \{1\}$. 
For two $\II$-labelings $(\iota,W)$ in $\Sigma\times\{0\}$ and $(\iota',W')$ in $\Sigma\times\{1\}$, we abbreviate $\Sk_\II(\Sigma \times [0,1],(\iota\sqcup\iota',(W,W'))$ as $\Sk_\II(\Sigma,(\iota,W),(\iota',W'))$. 
Note that if either $(\iota,W)$ or $(\iota', W')$ is admissible, then every compatible ribbon graph in $\Sigma \times [0,1]$ is automatically admissible. Similarly, in this case every skein relation is isotopic to an admissible one. 
The relative admissible skein module is then simply the usual relative skein module.
\begin{remark} 
We follow the definition of admissible skein modules of \cite{CGPAdmissbleskein} but allow non-empty boundary objects.
One difference though is that they allow some colors of the ribbon graphs to be in $\AA$, and only ask that at least one edge per connected component is in $\II$. 
The resulting notion is equivalent to ours by a standard trick. 
One can always run the $\II$-colored strand next to every other and fuse them. 
The resulting color is in $\II$ because $\II$ is a tensor ideal. 
This operation is readily checked to be well-defined up to admissible skein relations. 
Similarly our notion of admissible skein relation is easily checked to be equivalent to theirs, by isotoping an $\II$-colored strand outside the cube to intersect one of its faces.

\end{remark}
\begin{definition}\label{def:modified-skein-cat}
Let $\Sigma$ be a compact oriented surface, possibly with boundary, and $\II$ a tensor ideal of $\AA$. 
The \textbf{modified skein category $\SkCat_\II(\Sigma)$ of the surface $\Sigma$ with coefficients in $\II$} has:
\begin{description}
    \item Objects: Admissible $\II$-colored disks in $\Sigma$
    \item Morphisms: The set of morphisms from $(\iota,W)$ to $(\iota',W')$ is the admissible relative skein module $\Sk_\II(\Sigma;(\iota,W),(\iota',W'))$
    \item Composition: Vertical stacking (with any isotopic choice of smoothing at the gluing).
\end{description}
\end{definition}
When context permits, we will sometimes shorten $\SkCat_\II$ to $\SkCat$ and $\II$-colored ribbon graphs to ribbon graphs. 
Without loss of generality, we can assume that the embeddings underlying objects of $\SkCat(\Sigma)$ preserve orientation. 
Any object $(\iota,W)$ with orientation-reversing $\iota$ is canonically isomorphic to $(\bar\iota,\bar W)$, where $\bar\iota$ is the orientation-preserving mirror of $\iota$ and $\bar W$ is the dual of $W$. 
\begin{remark}\label{rmk:when-II-is-AA-skcat}
If $\II$ happens to contain the unit, i.e. $\II=\AA$, we recover the usual notion of skein categories. 
In this case every $\II$-labeling is isomorphic to an admissible one because any non-empty collection of embeddings with colors the monoidal unit behaves like the empty object. 
\end{remark}
Let us now discuss functoriality in $\Sigma$ of the $\II$-skein category. Functoriality of the usual 
$\SkCat$ construction is used heavily in works such as \cite{GJS,CookeExcision}.
The $(2,1)$-categories $\Bimod$ and $\Surf$ feature prominently in our discussions, and we recount their definitions here.

See \cite{Andre_1966,benabou1973distributeurs} for early works on $\Bimod$ and \cite[Ch. 7]{Borceux_1994} \cite[Ch. 5]{Loregian_2021} for textbook accounts, noting that \emph{distributor} and \emph{profunctor} are both terms for the 1-morphisms of $\Bimod$. 
\begin{definition}\label{def:Bimod}
The (2,1)-category $\Bimod$ of $\Bbbk$-linear categories and bimodules has
    \begin{description}
        \item Objects: small $\Bbbk$-linear categories $\CC,\DD$
        \item 1-Morphisms: $\CC$-$\DD$-bimodules, i.e. functors $\CC\otimes \DD^{op} \to \Vect$ or equivalently functors $\CC\to \Fun(\DD^{op},\Vect)$. Here $\Vect$ is the cocomplete category of not-necessarily finite dimensional vector spaces.
        The composition of two bimodules $F:\CC\otimes \DD^{op} \to \Vect$ and $G:\DD\otimes \EE^{op} \to \Vect$ is 
        given by the coend (see \cite[\S IX.6]{MacLane_1978}, \cite{Loregian_2021}) $$(G\circ F)(C,E) := \int^{D\in\DD} F(C,D)\otimes G(D,E)\ .$$

        The identity on $\CC$ is $\Hom_\CC(-,-)$.
        
        \item 2-Morphisms: natural isomorphisms
    \end{description}
    Note that the co-completeness of $\Vect$ is crucial in the above definition, as it guarantees the existence of the colimits used when composing 1-morphisms.
A symmetric monoidal structure is induced by cartesian product on objects and tensor product on Hom-spaces \cite[Section 7]{DayStreet1997}.
\end{definition}
$\Bimod$ embeds as a full subcategory in the symmetric monoidal bicategory $\Pr$ of presentable linear categories equipped with Kelly-Deligne tensor product, see \cite{BCJReflDualPr, GJS} for a modern introduction.
A category $\CC \in \Bimod$ is mapped to its presheaf category $\widehat \CC := \Fun(\CC^{op},\Vect)$.
There is likewise a symmetric monoidal inclusion $\Cat \to \Bimod$ given on objects by the identity and on functors $\CC\to \DD$ by post-composition with the Yoneda embedding $\DD\to \widehat\DD:= \Fun(\DD^{op},\Vect)$.

\begin{definition}
The (2,1)-category $\Surf$ of surfaces has:
    \begin{description}
        \item Objects: compact oriented surfaces, possibly with boundary and non-connected. The empty surface is permitted.
        \item 1-Morphisms: orientation-preserving embeddings (these are not required to send the boundary to the boundary).
        \item 2-Morphisms: isotopies considered up to higher isotopy. 
Going forward we will suppress higher isotopies. 
    \end{description}
    Note that $\Surf$ is symmetric monoidal under disjoint union.
\end{definition}

\begin{theorem}\label{theorem:skcats-a-2-functor}
    The assignment $\Sigma \mapsto \SkCat_\II(\Sigma)$ can be extended to a symmetric monoidal 2-functor $$\SkCat_\II : \Surf \to \Bimod.$$
\end{theorem}
We give a new proof in the $\II$-colored $\Bimod$ context, but first establish a useful technical result which allows us to preserve admissibility while decomposing $\II$-skeins. 
\begin{remark}\label{rmk:fake-empty}
Suppose $\II \neq \AA$.
The admissibility condition brings us to a major departure from the skein categories of \cite{WalkerNotes,JohnsonFreydHeisenbergPicture}. 
It is crucial that the statement above has target $\Bimod$ instead of $\Cat$.
This detail becomes conspicuous with the inclusion of the empty surface $\emptyset \hookrightarrow \Sigma$, which induces a map $\SkCat_\II(\emptyset) \to \SkCat_\II(\Sigma)$.
The skein category of the empty surface has a single element with endomorphism algebra the base field.
When working with $\Cat$, the induced map would be a functor which would send this one object to the empty object (i.e. no marked points) if it existed.
On the other hand, in $\Bimod$ it is the presheaf-valued functor which lands on the presheaf which \emph{would be} represented by the empty object if, again, $\SkCat_\II(\Sigma)$ had one.

The empty presheaf exists even while the empty object does not because our admissibility condition requires $\II$-labelings in each connected component of the 3-cobordism, but \emph{not each component of its boundary}.
Therefore we have a distinguished non-representable presheaf which plays the role of the empty object in various constructions: $$\begin{aligned}
        \Dist_\Sigma: \SkCat_\II(\Sigma)^{op} &\to \Vect\\
      X &\mapsto \Sk_\II(\Sigma;X,\emptyset)
    \end{aligned}$$

We thank David Jordan for this insight, which allows us to make sense of the inclusion of the empty set $\emptyset\inj \Sigma$ as inducing a morphism between $\II$-skein categories in $\Bimod$. 
Theorem \ref{theorem:skcats-a-2-functor} restricted to disks shows that $\II$ gives this way a unital $\Disk_2$-algebra in $\Bimod$, as opposed to a non-unital $\Disk_2$-algebra in $\Cat$.
A more general discussion on how to turn a non-unital algebra in $\Cat$ into a unital one in $\Bimod$ can be found in \cite{Stroinski}.
\end{remark}

\begin{lemma}\label{lemma:push_skeins_wherever}
    Let $M$ be a compact connected 3-manifold and $N\subseteq M$ a non-empty sub-manifold of dimension 2 or 3 with $\partial N \subseteq \partial M$. 
Define $\Sk_\II^N(M;X)$ the relative admissible skein modules where ribbon graphs must intersect $N$, and isotopies must preserve this property.
    Then the canonical map induces an isomorphism of vector spaces $$\Sk_\II^N(M;X) \tilde\to \Sk_\II(M;X)$$ for any $\II$-labeling $X$ in $\partial M$.
\end{lemma}

The main argument of the following proof appears in various forms throughout this paper. Note that rigidity is essential here, as we need to drag graphs along arbitrary paths.
\begin{proof}
We start with surjectivity.
Let $T\in \Sk_\II(M;X)$ and choose any isotopy representative $\oT$, not necessarily intersecting $N$. 
Choose any point $p$ on any edge of $\oT$ and any path $\gamma$ in $M\smallsetminus \oT$ starting from $p$ and going through $N$. 
Isotope $\oT$ by pulling a small neighborhood of $p$ along the path $\gamma$. 
The resulting ribbon graph $\oT_{\gamma}$ is a representative of $T$ in $\Sk_\II^N(M;X)$.

Next we show injectivity.
First note that for any other generic choice $p', \gamma'$ we can form the ribbon graph $(\oT_{\gamma})_{\gamma'}$ where we have pulled strands along both paths. 
Retracting either one of them gives isotopies $(\oT_{\gamma})_{\gamma'}\sim \oT_{\gamma}$ and $(\oT_{\gamma})_{\gamma'}\sim \oT_{\gamma'}$ through ribbon graphs intersection $N$, so $\oT_{\gamma}$ and $\oT_{\gamma'}$ represent the same element in $\Sk_\II^N(M;X)$.

Consider an isotopy $\varphi: \oT \Rightarrow \oT'$ between two ribbon graphs intersecting $N$, but possibly not preserving this property. 
By definition it is an ambient isotopy $\varphi = (\varphi_s:M\tilde\to M)_{s\in[0,1]}$ with $\varphi_0=\id_M$, $\varphi_1(\oT)=\oT'$, and $\oT_s:=\varphi_s(\oT)$ is a ribbon graph in $M$, possibly not intersecting $N$. 
Pick any $s\in [0,1]$ and choose a path $\gamma_s$ as above such that $(\oT_s)_{\gamma_s}$ intersect $N$ transversely at least once (this is why we need $N$ of codimension at most 1). 
This condition is open and holds for small perturbations of $(\oT_s)_{\gamma_s}$. 
By compactness we can pick finitely many times $s_i$ such that the isotopy $$(\oT_{s_i})_{\gamma_{i}} \sim (\oT_{s_{i+1}})_{\gamma'_{i}}$$ intersects $N$ throughout (here $\gamma'_i$ is the deformed $\gamma_i$). 
Now as explained above, there is an isotopy $(\oT_{s_{i+1}})_{\gamma'_{i}} \sim (\oT_{s_{i+1}})_{\gamma_{i+1}}$ that maintains at least one intersection with $N$. 
We have given a chain of isotopies from $\oT$ to $\oT'$ intersecting $N$ throughout, so any two ribbon graphs isotopic in $\Sk_\II(M;X)$ are likewise in $\Sk_\II^N(M;X)$. 
Now up to isotopy every skein relation can be confined to happen in a small ball, either contained in $N$ or disjoint from $N$, in particular not affecting the intersection with $N$. 
\end{proof}

With Lemma \ref{lemma:push_skeins_wherever} at our disposal, we can now prove the functoriality result:
\begin{proof}[Proof of Theorem \ref{theorem:skcats-a-2-functor}]
    Definition \ref{def:modified-skein-cat} gives the value on objects of $\Surf$. 
Disjoint union is sent to the tensor product in $\Bimod$ and the monoidal unit $\varnothing$ is sent to the one-object category with endomorphisms $\Bbbk$.
    The bulk of this proof is defining $\SkCat$ on embeddings and isotopies.
\paragraph{Embeddings to Bimodules.}
    Let $\kappa : \Sigma \to \Sigma'$ be an embedding.  We set
        $$\begin{array}{rcll}
        \SkCat(\kappa):\SkCat(\Sigma)& \otimes& \SkCat(\Sigma')^{op}  &\to \Vect\\
X &\otimes& Y &\mapsto \Sk_\II(\Sigma';Y,\kappa_*X)
        \end{array}$$
    where $\kappa_*(\iota,W) := (\kappa\circ \iota,W)$. 
Morphisms $T \subseteq \Sigma\times[0,1]$ and $T' \subseteq \Sigma'\times[0,1]$ transform the admissible relative skein module by post-composition with $\kappa_*T := (\kappa\times\id)(T)$ and pre-composition with $T'$.
    
If $\kappa$ is surjective on connected components (or if the tensor unit of $\AA$ is in $\II$) this bimodule is induced by a functor $\kappa_* : \SkCat(\Sigma) \to \SkCat(\Sigma')$.

We need to check that this construction preserves composition. 
Consider $\Sigma_1\overset{\kappa^1}{\inj}\Sigma_2\overset{\kappa^2}{\inj}\Sigma_3$. 
By definition, the composite bimodule on a pair of objects $X_1\in\SkCat(\Sigma_1)$ and $X_3\in\SkCat(\Sigma_3)$ is the coend 
\begin{equation}\label{eq:comp-of-embeddings}
    \int^{X_2\in\SkCat(\Sigma_2)} \Sk_\II(\Sigma_2;X_2, \kappa^1_*X_1) \otimes \Sk_\II(\Sigma_3;X_3, \kappa^2_*X_2)
\end{equation}
There's a canonical map from this composition to $\Sk_\II(\Sigma_3;X_3, (\kappa^2\kappa^1)_*X_1)$ which sends a simple tensor $T_2 \otimes T_3$ in \eqref{eq:comp-of-embeddings} to $\kappa^2_*T_2\circ T_3$. This is dinatural as $\kappa^2_*(T_2 \circ T)\circ T_3 = \kappa^2_*T_2\circ (\kappa^2_* T\circ T_3)$.
We must show that this defines an isomorphism.
We need to show that any skein $S\in \Sk_\II(\Sigma_3;X_3, (\kappa^2\kappa^1)_*X_1)$ factors as $$S=\kappa^2_*T_2\circ T_3\text{ for some }T_2\otimes T_3 \in \Sk_\II(\Sigma_2;X_2, \kappa^1_*X_1) \otimes \Sk_\II(\Sigma_3;X_3, \kappa^2_*X_2)\ ,$$ where $X_2$ is an admissible $\II$-labeling in $\Sigma_2$. This turns out to be a special case of Proposition \ref{prop:skeins_glue}.

We can always factor $S$ as $S = \id_{(\kappa^2\kappa^1)_*X_1}\circ S = \kappa^2_*(\id_{\kappa^1_*X_1})\circ S$, but ${\kappa^1_*X_1}$ may not be admissible. We need to pull a strand of $S$ along a path $\gamma$ to cross $\kappa_2(\Sigma_2)\times\{\frac12\}$ as in Lemma \ref{lemma:push_skeins_wherever}. Let us check that this is well-defined up to coend relations. The coend relations allow us to identify the two ways of splitting a skein $\kappa^2_*T_2\circ \kappa^2_*T\circ T_3$, into $\kappa^2_*T_2\circ \kappa^2_*T\otimes T_3$ or $\kappa^2_*T_2\otimes \kappa^2_*T\circ T_3$, i.e. to change the height where we cut. We want to decompose $S$ into two skeins each considered up to isotopy, so we consider everything up to isotopies fixing $\kappa_2(\Sigma_2)\times\{\frac12\}$.

As in Lemma \ref{lemma:push_skeins_wherever}, consider a different path $\gamma'$ used to push a different strand of $S$ into $\kappa_2(\Sigma_2)\times\{\frac12\}$. Call the resulting ribbon graphs $S_\gamma$ and $S_{\gamma'}$. We can assume $\gamma$ and $\gamma'$ are disjoint and consider the ribbon graph $(S_\gamma)_{\gamma'}$ obtained by pushing both strands through $\kappa_2(\Sigma_2)\times\{\frac12\}$. Then $S_\gamma$ is related to $(S_\gamma)_{\gamma'}$ by a coend relation, either cutting above or below the end of $\gamma'$, but always below the end of $\gamma$. Similarly, $(S_\gamma)_{\gamma'}$ is related to $S_{\gamma'}$ by a coend relation. 

\paragraph{Isotopies to Transformations.}
    Let $H: \kappa^0 \Rightarrow \kappa^1$ be an isotopy, i.e. a continuous map $H : \Sigma\times [0,1] \to \Sigma'$ with each $H(-,t)$ an embedding, $H(-,0) = \kappa^0$, and $H(-,1) = \kappa^1$. Tracing out the isotopy we obtain an embedding     
    \begin{equation}
    \begin{aligned}
        \phi_H : \Sigma \times [0,1] &\to \Sigma' \times [0,1] \\
        (p,t) &\mapsto \left( H(p,t), t\right).
    \end{aligned}
    \end{equation} 
An object $X \in \SkCat(\Sigma)$ gives an $\II$-colored ribbon graph $\phi_H(X\times[0,1])$ from $\kappa^0_*X$ to $\kappa^1_*X$ in $\Sigma'\times[0,1]$ whose strands are colored and oriented by the colors and orientations of $X$. Isotopic isotopies give isotopic ribbon graphs. 

The natural transformation $\SkCat(H) : \SkCat(\kappa^0) \Rightarrow \SkCat(\kappa^1)$ has components $\SkCat(H)_{X,Y}: \Sk_\II(\Sigma';Y,\kappa^0_*X) \to \Sk_\II(\Sigma';Y,\kappa^1_*X)$ induced by composition with $\phi_H(X\times[0,1])$.
    
Next we show naturality of $\SkCat(H)$. Naturality in $Y$ is clear as we compose in the $X$ direction. For naturality in $X$, we need to show that for any morphism $T : X_1 \to X_2$ in $\SkCat(\Sigma)$ and any $Y \in \SkCat(\Sigma')$, the following diagram commutes:
 \begin{equation}\label{eq:two-functor-proof-naturality}
     \begin{tikzcd}
        \SkCat(\kappa^0)(X_1,Y) \ar[r,"\SkCat(\kappa^0)(T)"]\ar[d,"\SkCat(H)_{X_1,Y}"'] &[1cm] \SkCat(\kappa^0)(X_2,Y) \ar[d,"\SkCat(H)_{X_2,Y}"] \\
        \SkCat(\kappa^1)(X_1,Y) \ar[r,"\SkCat(\kappa^1)(T)"] &[1cm] \SkCat(\kappa^1)(X_2,Y)
     \end{tikzcd}
 \end{equation}
These maps are induced by composition with respectively $(\kappa^1_*T)\circ \phi_H(X_1\times[0,1])$ and $\phi_H(X_2\times[0,1]) \circ (\kappa^0_*T)$. We will show that these two ribbon graphs are isotopic.
First, consider the isotopy $\Phi : \Sigma \times [0,1]\times [0,1] \to \Sigma' \times [0,1]$ given by
\begin{equation}\label{eq:two-functor-proof-isotopy}
 (p,t,s) \longmapsto   \begin{cases}
        (\kappa^1(p),t)&\text{ if } \frac{2}{3}s+\frac{1}{3} \leq t \\
        (H(p,3t-2s),t)&\text{ if } \frac{2}{3}s\leq t \leq \frac{2}{3}s + \frac{1}{3}\\
        (\kappa^0(p),t)&\text{ if } t\leq \frac{2}{3}s\\      
    \end{cases}
\end{equation}
At $s=0$, $\Phi$ applies the embedding $\phi_H$ in the bottom third of the thickened surface and $\kappa^1$ in the top two thirds.
At $s=1$ it applies $\kappa^0$ on the bottom two thirds and $\phi_H$ in the top third. 
At values in between it transitions from $\kappa^0$ to $\kappa^1$.

Next let $\overline{T}$ be any isotopy representative which is strictly vertical outside of $\Sigma\times [\frac{1}{3},\frac{2}{3}]$.
Then $\Phi(\cdot,0)|_{\overline{T}} = (\kappa^1_*T)\circ \phi_H(X_1\times[0,1])$ and $\Phi(\cdot,0)|_{\overline{T}} = \phi_H(X_2\times[0,1]) \circ (\kappa^0_*T)$ as morphisms in the skein category of $\Sigma'$. 
We conclude that \eqref{eq:two-functor-proof-naturality} commutes, and therefore that $\SkCat(H)$ is indeed a natural transformation.
\end{proof}

\subsection{Skein algebras}\label{sec:mod-skein-algebras}
The inclusion of the empty surface $\emptyset \inj \Sigma$ induces a morphism between $\II$-skein categories $\SkCat_\II(\emptyset) = \idty_{\Bimod} \to \SkCat_\II(\Sigma)$ in $\Bimod$, i.e. a presheaf-valued linear functor.
This gives a presheaf $$\Dist_\Sigma \in \widehat{\SkCat}_\II(\Sigma)$$ which we call the distinguished object.
See Remark \ref{rmk:fake-empty} for more details on $\Dist_\Sigma$.
\begin{definition}\label{def:I-skein-algebra}
    The \textbf{modified skein algebra of the surface $\Sigma$ with coefficients in $\II$} is the endomorphism algebra of the distinguished presheaf in the presheaf category of the $\II$-skein category of $\Sigma$:
    $$\SkAlg_{\II}(\Sigma) := \End_{\widehat{\SkCat}_\II(\Sigma)}(\Dist_\Sigma)$$ 
    More explicitly, an element $\alpha$ of the $\II$-skein algebra of $\Sigma$ is a collection of linear maps $$\alpha_X : \Sk_\II(\Sigma;X,\emptyset) \to \Sk_\II(\Sigma;X,\emptyset)$$
    natural in the admissible $\II$-labeling $X$. 
    The product of two elements $\alpha$ and $\beta$ is given by composing their components $(\alpha\cdot\beta)_X = \alpha_X \circ \beta_X$.
\end{definition}
We will see in Remark \ref{rmk:Comparison-CGP-algebras} that this notion differs from the non-unital admissible skein algebras of \cite[Prop. 2.5]{CGPAdmissbleskein}.

\subsubsection{Old and new skein algebras}\label{sec:less-is-more}
We now take some time to compare $\SkAlg_\II(\Sigma)$ and $\SkAlg_\AA(\Sigma)$.
When $\II=\AA$, Definition \ref{def:I-skein-algebra} recovers the usual notion of a skein algebra, since in this case the distinguished presheaf is represented 
by any collection of disks all colored by the monoidal unit.
By the Yoneda lemma, the associated skein algebra is isomorphic to its endomorphism algebra in the skein category\footnote{Recall from Remark \ref{rmk:when-II-is-AA-skcat} that the usual notion of skein category is likewise recovered when $\II = \AA$.}, which is the usual skein algebra $\SkAlg_\AA(\Sigma)$ of the surface $\Sigma$. 

The skein algebra for a proper ideal $\II \subsetneq \AA$ is not a restriction of the usual skein algebra to $\II$-colored strands.
On the contrary, every $\AA$-colored ribbon graph defines an element in the $\II$-skein algebra.
Not only that, but the $\II$-skein algebra appears to contain both elements and relations beyond those in the $\AA$-skein algebra.
These ``new skeins" are only traditionally defined in the presence of $\II$-colored strands, but can exist on their own in the $\II$-skein algebra.
    
    \begin{definition}\label{def:AA-to-II-skalg-map}
  The canonical algebra homomorphism $$\SkAlg_\AA(\Sigma) \to \SkAlg_\II(\Sigma)$$ takes a closed $\AA$-colored skein $T\in \SkAlg_\AA(\Sigma)$ to the natural transformation that acts on $\Sk_\II(\Sigma;-,\emptyset)$ by stacking $T$ on top.
  The resulting ribbon graph is always skein-equivalent to an admissible $\II$-colored one precisely because $\II$ is a tensor ideal.
  
  More generally, each inclusion of ideals $\II \subset \mathcal{J}$ induces an algebra map $\SkAlg_{\mathcal{J}}(\Sigma) \to \SkAlg_\II(\Sigma)$. 
\end{definition}

We will introduce the ``new" skeins, which the reader can convince themselves do not generally lie in the image of this map. 
The following definition follows \cite{DGGPR}, though the finiteness assumptions have been dropped.
\begin{definition}
Let $\II \subseteq \AA$ be an ideal in a ribbon category. A \textbf{bichrome graph} $T$ in a 3-manifold $M$ is a framed oriented embedded graph whose edges are either colored by an object of $\II$ or labeled ``red". We call the former ``blue" edges. Red edges are required to be loops whose endpoints are adjacent in the cyclic ordering on the half-edges ending at their shared vertex. Only blue edges can end at the boundary. See Figure \ref{fig:a-bichrome-graph}.

Vertices touching only blue edges are colored by morphisms in the usual way.
For the morphism coloring mixed vertices, we fuse each pair of adjacent red strands and color the fused pair by Lyubashenko's coend $\LL := \int^{X \in \II} X \otimes X^* \in \widehat \II$. We can then color the vertex in the usual way but by a morphism in $\widehat\II$, using the Yoneda embedding $\II \inj \widehat\II$ as necessary. 

The bichrome graph $T$ is called \textbf{admissible} if there is at least one blue edge per connected component of $M$.
\end{definition}
\begin{figure}[h]
    \centering
\begin{tikzpicture}
    \node[inner sep = 0pt, outer sep = 0pt] (dot) at (0,0) {$\bullet$};
    \draw[red] (dot).. controls (-1,0.5).. (-1,1)node[midway, sloped]{$\scriptstyle <$};
    \draw[red] (dot).. controls (-0.5,0.5)..(-0.5,1);
    \draw[red] (dot)--++(0,1)node[midway, sloped]{$\scriptstyle >$};
    \draw[red] (dot).. controls (0.5,0.5)..(0.5,1);
    \draw[blue] (dot).. controls (1,0.5)..(1,1) node[midway, sloped]{$\scriptstyle <$}node[midway, right]{$P$};
    \draw[red, dashed] (-1,1) arc (180:0:0.25 and 1);
    \draw[red, dashed] (0,1) arc (180:0:0.25 and 1);
    \draw[blue, dashed] (1,1) -- ++(0,1);
\end{tikzpicture}    
\caption{A mixed vertex in a bichrome graph. It is colored by a morphism $f: \idty \to \LL\otimes \LL \otimes P^*$ in $\widehat \II$, i.e. a natural transformation from $\Hom_\AA(-,\idty)$ to $\int^{X,Y\in\II}\Hom(-,X\otimes X^*\otimes Y \otimes Y^* \otimes P^*)$.}    
\label{fig:a-bichrome-graph}
\end{figure}

The admissibility condition allows us to turn bichrome graphs into $\II$-colored (blue) ones.
\begin{definition} The \textbf{red-to-blue operation} associates an $\II$-skein $\widetilde T$ to an admissible bichrome graph $T$ in $M$. We describe the procedure for a single mixed vertex.

To simplify our equations, we assume all red edges are adjacent.
First pull a blue strand colored by  $Q \in \II$ near the vertex. 
Locally, we observe the evaluation $P := Q \otimes Q^* \to \idty$ and the original vertex $\idty \to \LL^{\otimes k} \otimes X$ where $X = \idty$ or $X \in \II$.
Composing we get a morphism $f:P \to \LL^{\otimes k} \otimes X$ in $\widehat\II$. We have a natural isomorphism
\begin{equation}\label{eq:red-to-blue-lift}
    \Hom_{\widehat\II}(P, \LL^{\otimes k}\otimes X) \simeq \int^{X_1,\dots,X_k \in \II} \Hom_\II(P,X_1\otimes X_1^*\cdots X_k\otimes X_k^* \otimes X)
\end{equation}
coming from the fact that $P$ is compact-projective in $\widehat\II$ (said differently, that colimits in the presheaf category $\widehat\II$ are computed point-wise).
Choose a representative $\sum_j \lambda_j f_j$ of $f$ in the coend\footnote{The right hand side of \eqref{eq:red-to-blue-lift} is a colimit in vector spaces, so is a quotient of the direct sum of its components, which are Hom spaces.}, with each $f_j : P \to X_{j,1}\otimes X_{j,1}^*\cdots X_{j,k}\otimes X_{j,k}^* \otimes X$ a morphism in $\II$.

Finally, define $\widetilde T := \sum_j \lambda_j \widetilde T_j$ where $\widetilde T_j$ is the $\II$-colored ribbon graph where the $i$-th red strand is colored by $X_{j,i}$ and the mixed vertex is colored by $f_j$. See Figure \ref{fig:red-to-blue}.
\end{definition}
\begin{figure}[h]
    \centering
\begin{tikzpicture}[baseline = 0pt]
    \node[inner sep = 0pt, outer sep = 0pt] (dot) at (0,0) {$\bullet$};
    \draw[red] (dot).. controls (-1,0.5).. (-1,1)node[midway, sloped]{$\scriptstyle <$};
    \draw[red] (dot).. controls (-0.5,0.5)..(-0.5,1);
    \draw[red] (dot)--++(0,1)node[midway, sloped]{$\scriptstyle >$};
    \draw[red] (dot).. controls (0.5,0.5)..(0.5,1);
    \draw[blue] (1.5,1).. controls (1.5, 0.5) and (1,0.5)..(1,1) node[midway, sloped]{$\scriptstyle <$}node[midway, right]{$Q$};
    \draw[red, dashed] (-1,1) arc (180:0:0.25 and 1);
    \draw[red, dashed] (0,1) arc (180:0:0.25 and 1);
    \draw[blue, dashed] (1,1) -- ++(0,1);
    \draw[blue, dashed] (1.5,1) -- ++(0,1);
\end{tikzpicture}    
=
\begin{tikzpicture}[baseline = 0pt]
    \node[inner sep = 0pt, outer sep = 0pt] (dot) at (0,0) {$\bullet$};
    \draw[red] (dot).. controls (-1,0.5).. (-1,1)node[midway, sloped]{$\scriptstyle <$};
    \draw[red] (dot).. controls (-0.5,0.5)..(-0.5,1);
    \draw[red] (dot)--++(0,1)node[midway, sloped]{$\scriptstyle >$};
    \draw[red] (dot).. controls (0.5,0.5)..(0.5,1);
    \draw[red, dashed] (-1,1) arc (180:0:0.25 and 1);
    \draw[red, dashed] (0,1) arc (180:0:0.25 and 1);
    \draw[blue, dashed] (1,1) -- ++(0,1);
    \draw[blue, dashed] (1.5,1) -- ++(0,1);
    \node[inner sep = 0pt, outer sep = 0pt] (dot) at (0,-0.3) {$\bullet$};
    \draw[blue] (1.5,1).. controls (1.5, -1.3) and (-0.3,-1)..(0,-0.3) node[midway, sloped]{$\scriptstyle <$}node[midway, below right = -3pt]{$Q$};
    \draw[blue] (1,1).. controls (1,-0.5) and (0.3,-1)..(0,-0.3) node[midway, sloped]{$\scriptstyle >$};
    \draw[dashed] (-0.4,-0.5) rectangle (0.4,0.2);
    \node at (-0.6,-0.15) {$f$};
\end{tikzpicture}    
$\quad \quad \leadsto\quad \quad$
\begin{tikzpicture}[baseline = 0pt]
    \node[inner sep = 0pt, outer sep = 0pt] (dot) at (0,0) {$\bullet$};
    \draw[blue] (dot).. controls (-1,0.5).. (-1,1)node[midway, sloped]{$\scriptstyle <$} node[midway, left]{$X$};
    \draw[blue] (dot).. controls (-0.5,0.5)..(-0.5,1);
    \draw[blue] (dot)--++(0,1)node[midway, sloped]{$\scriptstyle >$} node[midway, left=-2pt]{$Y$};
    \draw[blue] (dot).. controls (0.5,0.5)..(0.5,1);
    \draw[blue, dashed] (-1,1) arc (180:0:0.25 and 1);
    \draw[blue, dashed] (0,1) arc (180:0:0.25 and 1);
    \draw[blue, dashed] (1,1) -- ++(0,1);
    \draw[blue, dashed] (1.5,1) -- ++(0,1);
    \draw[blue] (1.5,1).. controls (1.5, -1.3) and (-0.3,-1)..(0,0) node[midway, sloped]{$\scriptstyle <$}node[midway, below right=-3pt]{$Q$};
    \draw[blue] (1,1).. controls (1,-0.5) and (0.3,-1)..(0,0) node[midway, sloped]{$\scriptstyle >$};
    \node at (-0.25,-0.25) {$\tilde f$};
\end{tikzpicture}    
\caption{The red-to-blue operation. Here $\tilde f : P \to X \otimes X^*\otimes Y \otimes Y^*$ is a representative of $f: P \to \LL \otimes \LL$, and $P:=Q\otimes Q^*$.}    
\label{fig:red-to-blue}
\end{figure}

A number of choices must be made in the red-to-blue operation, but the result is nevertheless well defined. 
\begin{lemma}
Two skeins obtained via two different red-to-blue operations on the same admissible bichrome graph $T$ are related by isotopies and admissible skein relations.
\end{lemma}
\begin{proof}
There are two choices in the red-to-blue operation.
First, there is the choice of a blue edge and the path used to to pull it near the mixed vertex.
Second there is the choice of a representative for $f$ in the right hand side of \eqref{eq:red-to-blue-lift}.
This is unique up to the coend relations, which leave the skein invariant by isotoping vertices along the red edges (recall that we assume that pairs of adjacent red half-edges are connected.)

For path independence we do the usual trick of bringing both chosen blue edges near the vertex.
Suppose one is labeled $Q_1$ and the other $Q_2$.
Denote $f_1$ (resp. $f_2$) the morphisms in \eqref{eq:red-to-blue-lift} obtained by pulling the first (resp. second) edge along its chosen path.
After pulling both edges, locally we observe a tensor product of evaluations and the original vertex. Their composition $f_{12}$ is equal to both $f_1 \otimes ev_{Q_2}$ and $ev_{Q_1} \otimes f_2$. A choice of representative for either $f_1$ or $f_2$ provides a representative for $f_{12}$. The two red-to-blue modifications differ only by the choice of a representative for $f_{12}$, and hence do not affect $\widetilde T$.
\end{proof}
Using the operation $T \mapsto \widetilde T$ we can now define skeins from bichrome graphs.
\begin{definition}
    The \textbf{modified skein $\alpha^T\in\SkAlg_\II(\Sigma)$ associated with a bichrome graph} $T$ is the natural transformation whose components $$\alpha^T_X: \Sk_\II(\Sigma;X,\emptyset) \to \Sk_\II(\Sigma;X,\emptyset)$$
    are given by stacking $T$, then turning the resulting admissible bichrome graph blue using the red-to-blue operation. This is natural because the bichrome graph obtained after stacking is natural and the red-to-blue operation is well-defined.
\end{definition}
In the definition we've used that non-admissible bichrome graphs will become admissible after stacking on an element of $\Sk_\II(\Sigma;X,\emptyset)$.
Note that if $T$ is an admissible bichrome graph, it can be turned blue before stacking and $\alpha^T$ is the natural transformation coming from stacking the $\II$-colored ribbon graph $\widetilde T$.
If $T$ is not admissible (i.e. has no blue edges) then $\alpha^T$ will not in general be the image of an $\AA$-colored skein.

The take-away here is that there are more ways of acting (naturally) on $\II$-skeins than there are of acting on all skeins, and therefore the $\II$-skein algebra is bigger than the $\AA$-skein algebra. 

Using factorization homology (see Corollary \ref{cor:ModSkAlgIsInvCoend}) one can prove that the $\alpha^T$ generate the $\II$-skein algebra for surfaces with non-empty boundary.

\begin{remark}
    The red-to-blue operation builds on Lyubashenko's way of interpreting surgery on a 3-manifold, replacing the Kirby color used in the semisimple case.
    It already appeared in \cite[Lemma 4.5]{DGGPR} in the case when $\II = \operatorname{Proj}(\VV) \subseteq \VV$ is the ideal of projectives in a finite unimodular ribbon tensor category $\VV$.
    Note that we write $\LL$ as a coend over $\II$ instead of $\AA$, so that we get $\II$-colored ribbon graphs from the summands.
    The two coends are canonically isomorphic \cite[Prop. 5.1.7]{KerlerLyubBook} when $\II$ is the ideal of projective objects in a finite tensor category.
\end{remark}

\begin{remark}\label{rmk:Comparison-CGP-algebras}
Admissible skein algebras whose elements are admissible closed ribbon graphs in the thickened surface were introduced in \cite[Proposition 2.5]{CGPAdmissbleskein}. They do not allow the empty skein and are therefore non-unital.
There is a canonical algebra map from this admissible skein algebra to the usual $\AA$-skein algebra given by inclusion.
Composing with the canonical morphism $\SkAlg_\AA(\Sigma) \to \SkAlg_\II(\Sigma)$ of Definition \ref{def:AA-to-II-skalg-map}, we see that the $\II$-skein algebra contains but is not generated by elements coming from the admissible skein algebra.
Consider for example the empty skein.
We will see that all three algebras act on admissible skein modules.
These algebra maps preserve the action and therefore the actions of the admissible and $\AA$-skein algebras can be recovered by that of the $\II$-skein algebra.
\end{remark}


\subsection{Skein modules of 3-cobordisms}\label{sec:mod-skein-modules}
As usual, a 3-dimensional cobordism $M$ between two surfaces induces a bimodule between their $\II$-skein categories.
To see how this works, we take a moment to consider how orientation affects the skein category.
Let $\overline\Sigma$ be $\Sigma$ with opposite orientation.
There is a categorical equivalence $\SkCat_\II(\Sigma) \simeq \SkCat_\II(\overline\Sigma)^{op}$ induced by the orientation reversing diffeomorphism \begin{equation*}
        \begin{aligned}
           \operatorname{rev}: \Sigma \times [0,1] &\to \Sigma \times [0,1]\\
            (x,t)&\mapsto (x,1-t)\quad .
        \end{aligned}
    \end{equation*}
An object $(\iota,W)$ is mapped to itself (though now $\iota$ is orientation reversing) and a ribbon graph $T\subseteq \Sigma\times[0,1]$ is mapped to $\operatorname{rev}(T)\subseteq \overline\Sigma\times[0,1]$ which now goes from the target of $T$ back to its source.
\begin{definition}\label{def:adm-skein-bimodule-functor}
    Let $M: \Sigma' \to \Sigma$ be a cobordism. That is, $M$ is a compact oriented 3-manifold with a diffeomorphism $\partial M \simeq \Sigma \sqcup \overline\Sigma'$ where $\partial M$ is oriented with outward normal. 
The \textbf{admissible skein bimodule} of $M$ is the functor \begin{equation}
    \begin{aligned}
\SkFun_\II(M): \SkCat_\II(\Sigma)\otimes\SkCat_\II(\Sigma')^{op} &\to \Vect\\
(X,Y)&\mapsto \Sk_\II(M;X,Y)
    \end{aligned}
        \end{equation}
The action of morphisms in $\SkCat_\II(\Sigma)$ and $\SkCat_\II(\Sigma')$ is induced by stacking in a neighborhood of $\partial M$ diffeomorphic to $(\Sigma \sqcup \overline\Sigma')\times[0,1]$.
\end{definition}
Note that $M$ above is a cobordism from $\Sigma'$ to $\Sigma$, whereas its skein bimodule goes in the other direction.
This contravariance also appears in \cite{WalkerNotes}, where skeins are treated as the dual theory to a TQFT.
We emphasize that this is only a nuisance and not a deep issue, since $\operatorname{Cob} \simeq \operatorname{Cob}^{op}$ via orientation reversal. 

The results of \cite[Section 4.4]{WalkerNotes} generalize to this context.
\begin{definition}
Let $M$ be a compact oriented 3-manifold, $\Sigma$ a compact oriented surface, possibly with boundary, and $\Sigma\sqcup-\Sigma\overset{\kappa\sqcup\iota}\inj \partial M$ an embedding. Denote $M_{gl} = M/{\scriptstyle \kappa(x)\sim\iota(x)}$ the compact oriented 3-manifold obtained by identifying $\Sigma$ to $-\Sigma$, and call $\pi:M\to M_{gl}$ the quotient map. 

Let $X$ be an $\II$-labeling in $\partial M$ disjoint from $-\Sigma\sqcup\Sigma$ which we will also see as an $\II$-labeling of $M_{gl}$, and let $Y$ be an $\II$-labeling of $\Sigma$. We denote the $\II$-labeling of $M$ obtained as the disjoint union of $X$, $\kappa(Y)$ and $\iota(Y)$ by $X\sqcup Y\sqcup -Y$.

The \textbf{gluing along $\Sigma$} of a skein $T\in \Sk_\II(M;X\sqcup Y\sqcup-Y)$ is the skein 
$$\mathrm{gl}_\Sigma(T) = \pi(T) \in \Sk_\II(M_{gl};X)$$
where $\pi(T)$ has colored inherited from $T$, and the new interior points corresponding to the endpoints $Y$ are colored by identity morphisms. Clearly, isotopic and skein equivalent skeins give isotopic and skein equivalent gluings.
\end{definition}
\begin{proposition}\label{prop:skeins_glue}
    Let $M$, $\Sigma$, $M_{gl}$ and $X$ be as above. Then gluing along $\Sigma$ induces an isomorphism of vector spaces 
    \begin{equation}\label{eq:skeins_glue}
        \mathrm{gl}_\Sigma:\int^{Y \in \SkCat_\II(\Sigma)} \Sk_\II(M;X\sqcup Y\sqcup -Y) \tilde\to \Sk_\II(M_{gl},X)\ .
    \end{equation}
\end{proposition}
\noindent 
We're again relying on the co-completeness of $\Vect$ when assuming the existence of the coend in \eqref{eq:skeins_glue}.
\begin{proof}
    This is \cite[Thm 4.4.2]{WalkerNotes} adapted to the non-semisimple setting. We suppose $\Sigma$ is connected for simplicity, one can apply the connected case iteratively on connected components to obtain the general result.

    First, note that by Lemma \ref{lemma:push_skeins_wherever} on $N = \Sigma \subseteq M_{gl}$ we may replace $\Sk_\II(M_{gl},X)$ with the space of skeins that intersect $\Sigma$ at least once and their isotopies. 

    The rest of the proof follows \cite[Thm 4.4.2]{WalkerNotes} directly, which we recall for the reader's convenience. The map $\mathrm{gl}_\Sigma$ is well-defined on the coend by considering the isotopy $\varphi$ that shifts in the $I$-direction in a collar $\Sigma\times I \inj M_{gl}$ of $\Sigma$ in $M_{gl}$. Up to an isotopy, every skein in $M_{gl}$ can be made transverse to $\Sigma$ while intersecting it, hence the map $\mathrm{gl}$ is surjective. Let us turn to injectivity.
    
    Fix two ribbon graphs $\oT$ and $\oT'$ in $M$ with boundary $X\sqcup Y\sqcup -Y$ and $X\sqcup Y'\sqcup -Y'$ representing skein $T$ and $T'$. Denote $\mathrm{gl}_\Sigma(\oT)$ and $\mathrm{gl}_\Sigma(\oT')$ the ribbon graphs (not up to isotopy) in $M_{gl}$ obtained by gluing them along $\Sigma$. Suppose that $\mathrm{gl}_\Sigma(T)=\mathrm{gl}_\Sigma(T') \in \Sk_\II(M_{gl};X)$ are equal, i.e. $\mathrm{gl}_\Sigma(\oT)$ and $\mathrm{gl}_\Sigma(\oT')$ can be related by isotopies and skein relations. Any isotopy can be decomposed into isotopies supported in small enough balls. Isotopies supported in balls disjoint from $\Sigma$ can be lifted in $M$ and we obtain an isotopy from $\oT$ to $\oT'$. If $\mathrm{gl}_\Sigma(\oT)$ and $\mathrm{gl}_\Sigma(\oT')$ are related by the isotopy $\varphi$ above, then $T$ and $T'$ are related by a coend relation. Up to conjugation with $\varphi$, an isotopy supported in a small enough ball intersecting $\Sigma$ can be made supported on a ball disjoint from $\Sigma$. Finally, any skein relation can be made supported in a ball disjoint from $\Sigma$ by an isotopy, in which case it comes from a skein relation in $M$. Hence, two ribbon graphs $\oT$ and $\oT'$ can be related by isotopies, skein relations and coend relations, and the map $\mathrm{gl}$ is injective.
\end{proof}

\begin{theorem}\label{thm:its-a-TQFT}
    Skein bimodules extend skein categories to a contravariant symmetric monoidal functor $$\SkFun_\II: \operatorname{Cob}_{2+1}\to \Bimod$$    
    i.e. skein categories and skein bimodules form a categorified 3-TQFT.
\end{theorem}
\begin{proof}
    Skein bimodules compose by Proposition \ref{prop:skeins_glue}, i.e. for any composable pair of cobordisms $\Sigma_1 \overset{M_{12}} \to \Sigma_2 \overset{M_{23}}\to \Sigma_3$ and objects $X_1, X_3$ of $\SkCat(\Sigma_1), \SkCat(\Sigma_3)$, the gluing map is an isomorphism:
    \begin{equation}\label{eq:skfun-composes}
            \int^{X_2 \in \SkCat(\Sigma_2)} \Sk_\II(M_{12};X_1,X_2) \otimes \Sk_\II(M_{23};X_2,X_3) \to \Sk_\II(M_{12}\underset{\Sigma_2}{\cup}M_{23};X_1,X_3).
    \end{equation}
Symmetric monoidality was already observed in Theorem \ref{theorem:skcats-a-2-functor} on objects and follows from $$\Sk_\II(M\sqcup M', X\sqcup X') \simeq \Sk_\II(M,X)\otimes \Sk_\II(M',X')$$ on morphisms.
\end{proof}

When $\II=\AA$, the usual skein module with empty boundary object is obtained by evaluating the skein bimodule on the empty object.
When $\II \subsetneq\AA$ is a proper ideal, we again generalize the construction by letting the distinguished presheaf play the role of the empty object.
\begin{definition}
    Let $M: \Sigma' \to \Sigma$ be a cobordism as above.
    The \textbf{modified skein module $\SkMod_\II(M)$ of $M$ with coefficients in $\II$} is the $\SkAlg_\II(\Sigma)$-$\SkAlg_\II(\Sigma')$-bimodule $$\SkMod_\II(M) := \Hom_{\widehat{\SkCat}_\II(\Sigma')}(\Dist_{\Sigma'},\widehat{\SkFun}(M)(\Dist_\Sigma))$$ where $\widehat{\SkFun}(M)$ denotes the essentially unique cocontinuous extension of $\SkFun(M):{\SkCat}_\II(\Sigma) \to \widehat{\SkCat}_\II(\Sigma')$ to $\widehat{\SkFun}(M): \widehat{\SkCat}_\II(\Sigma) \to \widehat{\SkCat}_\II(\Sigma')$, see \cite{BCJReflDualPr} or \cite[Prop. 2.2.4]{DuggerSheavesHomotopy}.
\end{definition}
Let's unpack this definition.
 If $M = \Sigma\times [0,1]$ is a cylinder, then $\widehat{\SkFun}(M)$ is the identity and $\SkMod_\II(M)$ coincides with the $\II$-skein algebra introduced before.

In general the presheaf $\widehat{\SkFun}(M)(\Dist_{\Sigma}) : \SkCat_{\II}(\Sigma')^{op} \to \Vect$ is given by the coend
\begin{equation}\label{eq:coYoneda-skmod}
    Y \mapsto \int^{X \in \SkCat_\II(\Sigma)} \hspace{-1em} \Sk_\II(M;X,Y)\otimes \Sk_\II(\Sigma\times[0,1]; X, \emptyset) \simeq \Sk_\II(M;\emptyset,Y)
\end{equation}
where the last equivalence uses the fact that \eqref{eq:skeins_glue} is an isomorphism.

An element $\alpha$ of $\SkMod_\II(M)$ is therefore a natural transformation $\alpha : \Dist_{\Sigma'} \Rightarrow \Sk_\II(M;-,\emptyset)$ determined by its coefficients
\begin{equation}
\alpha_Y : \Dist_{\Sigma'}(Y) := \Sk_\II(\Sigma'\times[0,1]; Y , \emptyset) \to \Sk_\II(M; \emptyset, Y) 
\end{equation}
which are natural in $Y \in \SkCat_\II(\Sigma')$.

The action of $\SkAlg_\II(\Sigma') := \End_{\widehat{\SkCat}_\II(\Sigma')}(\Dist_{\Sigma'})$ is given by composition in $\widehat{\SkCat}_\II(\Sigma')$.
The action of an element $\beta$ in $\SkAlg_\II(\Sigma)$ is given by post-composition with $\widehat{\SkFun}_\II(M)(\beta)$.

The empty skein in the $\II$-skein module of $M$ makes sense if and only if the outgoing boundary $\Sigma' \inj M$ is surjective on connected component. 
This is closely related to the fact that the undecorated \cite{DGGPR} theories are defined only for these cobordisms. 
Algebraically, it is an instance of the fact that the inclusion of the unit in a non-semisimple ribbon category is almost, but not completely, 3-dualizable \cite{HaiounUnit}. 

\begin{remark}\label{rmk:mod-SkAlg-acts-on-adm-SkMod} Skein modules as we've defined them treat incoming and outgoing boundary components very differently.
We consider a compact 3-manifold $M$ as a cobordism $\Sigma' \to \Sigma$ in two canonical ways. 

Taking $\Sigma' = \emptyset,\ \Sigma = \partial M$, $\SkMod_\II(M)$ is the admissible skein module, where we do not allow the empty skein. Hence the $\II$-skein algebra of $\partial M$ also acts on the admissible skein module of $M$ introduced in \cite{CGPAdmissbleskein}.
When $\Sigma' = \emptyset$, then $\widehat{\SkCat}_\II(\Sigma') = \Vect$ and $\Dist_{\Sigma'} = \Bbbk$. 
By \eqref{eq:coYoneda-skmod} we have $$\SkMod_\II(M)\simeq \int^{X \in \SkCat_\II(\Sigma)}\Sk_\II(\Sigma;X, \emptyset)\otimes \Sk_\II(M;X) \simeq \Sk_\II(M;\emptyset).$$

On the other hand, we can take $\Sigma' = \overline{\partial M}$ and $\Sigma = \emptyset$.
If each connected component of $M$ has non-empty boundary, then any $\AA$-colored skein in $M$ gives an element of the $\II$-skein module and there are new elements coming from the bichrome graphs described above.
\end{remark}

\section{Relation to factorization homology}
\newcommand{\Idiskalg}{I}
\newcommand{\diskalg}{E}

The main result of this section is Theorem \ref{thm:main-result}, which establishes an identification of $\SkCat_\II(\Sigma)$ with the factorization homology over $\Sigma$ of the $\Bimod$ disk algebra $\Idiskalg$ associated to $\II$.

\subsection{Factorization homology}\label{sec:factorization-homology}

Factorization homology is defined in \cite[Def 3.7]{Ayala_Francis_2019} for any $\Disk_n$-algebra with values in a symmetric monoidal $(\infty,1)$-category.
We will be interested in $\Disk_2$-algebras in $\Bimod$. 
As $\Bimod$ is a $(2,1)$-category, it is enough to consider the homotopy $(2,1)$-categories of disks and surfaces.
\begin{definition}
    The (2,1)-category $\Disk$ of disks is the monoidal full subcategory of $\Surf$ with objects given by finite disjoint unions of the standard disk. 
A \textbf{$\Disk$-algebra} is a symmetric monoidal 2-functor out of $\Disk$.
\end{definition}

\begin{definition}\label{def:factorization-homology}
    The \textbf{factorization homology} of a $\Disk$-algebra $\diskalg : \Disk \to \Bimod$ over a surface $\Sigma$ is the homotopy 2-colimit
    \begin{equation}
        \int_{\Sigma}\! \diskalg := 2\text{--}\colim \tilde{\diskalg} .
    \end{equation}
    Where $\tilde{\diskalg} : \Disk_{/\Sigma} \to \Bimod$ is the composite functor $\Disk_{/\Sigma} \overset{\text{for}}{\to} \Disk \overset{\diskalg}{\to} \Bimod$ and $\Disk_{/\Sigma}$ is the slice (2,1)-category.
    
By \cite[Lemma 4.3.2.13]{LurieHTT}, \cite[Proposition 3.9]{Ayala_Francis_2019} and using the universal property of colimits, this assignment extends uniquely to a symmetric monoidal 2-functor $\displaystyle \int_-\!\diskalg:\Surf \to \Bimod$.
\end{definition}
Let's unpack this definition, taking advantage of the fact that we're dealing with (2,1)-categories. 
The slice (2,1)-category $\Disk_{/\Sigma}$ has:
    \begin{description}
        \item Objects: disks over $\Sigma$, i.e. embeddings $\iota : d \hookrightarrow \Sigma$ where $d \in \Disk.$
        \item 1-Morphisms: embeddings of disks over $\Sigma$, i.e. 
        \begin{equation*}
        \begin{tikzcd}
            d \ar[rr,"\rho"]\ar[dr, "\iota"', ""{name=iota, above right}] &\ & d' \ar[dl, "\iota'"] \\
            &\Sigma
                    \ar[from=iota,to=1-3,Rightarrow,"\varphi"' near start, shorten > = 15]
        \end{tikzcd}.    
        \end{equation*}
        \item 2-Morphisms from $(\rho,\varphi)$ to $(\rho',\varphi')$: isotopies over $\Sigma$, i.e. 
        \begin{equation*}
            \begin{tikzcd}
            d \ar[rr, bend left,"\rho", ""{name=rho, below}]\ar[rr, bend right,"\rho'"', ""{name=rhoprime, above}] &\ & d'                     \ar[from=rho,to=rhoprime,Rightarrow,"h"]
        \end{tikzcd}
        \end{equation*} such that $(\iota'h)\circ\varphi=\varphi'$ as 2-morphisms in $\Surf$, i.e. up to higher isotopy.
    \end{description}
For any category $X$, let $\const_X : \Disk_{/\Sigma} \to \Bimod$ be the functor that assigns $X$ to every object, the identity to every 1-morphism, and equality to every 2-morphism.

By definition, the 2-colimit of the functor\footnote{By the colimit of a functor, we mean the colimit over the diagram defined by its image in the target category. e.g. there is an arrow in the diagram for each morphism in the source category.} $\tilde \diskalg$ is the data of a category $\int_\Sigma \diskalg$ together with a universal cocone.
In other words we have a strong natural transformation $c:\tilde \diskalg \Rightarrow \const_{\int_\Sigma \diskalg}$ which induces a categorical equivalence for each $X\in \Bimod$:
\begin{equation}\label{eq:colimit-equiv}
    \Hom_{\Bimod}(\int_\Sigma \diskalg, X)\overset{\sim}{\to}\Nat(\tilde{\diskalg}, \const_X).
\end{equation}
Here the left hand side is the groupoid of bimodules and natural isomorphisms, while the right hand side is the groupoid of strong natural transformations and modifications (also sometimes called transfors or 2-transfors), which we recall in Appendix \ref{sec:modifications}. 
The equivalence \eqref{eq:colimit-equiv} sends a morphism $G: \int_\Sigma \diskalg \to X$ to the composition $$\tilde{\diskalg} \overset{c}{\Longrightarrow} \const_{\int_\Sigma \diskalg} \overset{\const_G}{\Longrightarrow} \const_X\ ,$$ which is a cocone over the slice category with tip $X$, see Figure \ref{fig:coconefunctor}.
\begin{figure}
\begin{center}
    \begin{tikzcd}
        \diskalg(d_1) \ar[dr]\ar[dd]\ar[rrr,crossing over,"c_{\iota_1}"]&&     &[2em]
        \int_\Sigma \diskalg \ar[dr,equal]\ar[dd,equal]\ar[rrr,crossing over,"G"]&&  &[1em]  
        X \ar[dr,equal]\ar[dd,equal]&&          \\
        & \diskalg(d_2)\ar[rrr,crossing over,"c_{\iota_2}"]&                  &[2em]
        & \int_\Sigma \diskalg\ar[rrr,crossing over,"G"] &               &[1em]
        & X &                       \\
        \diskalg(d_3) \ar[ur]\ar[rrr,crossing over,"c_{\iota_3}"]&&            &[2em]
        \int_\Sigma \diskalg \ar[ur,equal]\ar[rrr,crossing over,"G"]&&         &[1em]
        X \ar[ur,equal]&&
    \end{tikzcd}
\end{center}
\caption{The natural transformation $\tilde \diskalg \overset{c}{\Rightarrow} \const_{\int_{\Sigma'}\!\diskalg}\overset{\const_G}{\Longrightarrow}\const_X$, with all 2-morphisms suppressed.}
    \label{fig:coconefunctor}
\end{figure}

\begin{remark}\label{rmk:2funinducecocone}
For any 2-functor $F : \Surf \to \Bimod$, we have a canonical cocone $c_F$ over the slice category $\Disk_{/\Sigma}$ with tip $F(\Sigma)$, i.e. a natural transformation $c_F: (F|_{\Disk} \circ \text{for}) \Rightarrow \const_{F(\Sigma)}$. 
Its component on an object $(\iota:d\to \Sigma)$ is the 1-morphism $(c_F)_\iota = F(\iota): F(d)\to F(\Sigma)$ in $\Bimod$ and its component on a 1-morphism $\Phi=(\rho,\varphi):\iota \to \iota'$ of $\Disk_{/\Sigma}$ is $$(c_F)_{\Phi} =\begin{tikzcd}
    F(d) \ar[d, "F(\rho)"'] \ar[r, "F(\iota)"] & F(\Sigma) \ar[d, equal]\\
    F(d') \ar[r, "F(\iota')"'] \arrow[ur, Rightarrow, "F(\varphi)"] & F(\Sigma)
\end{tikzcd}$$
A symmetric monoidal 2-functor $F:\Surf \to \Bimod$ coincides with factorization homology of $F\vert_{\Disk}$ if and only if each cocone $c_F: (F|_{\Disk} \circ \text{for}) \Rightarrow \const_{F(\Sigma)}$ induces an equivalence $$\Hom_{\Bimod}(F(\Sigma), X)\overset{\sim}{\to}\Nat(F\vert\Disk, \const_X)$$ for every $X$ in $\Bimod$ and $\Sigma$ in $\Surf$.
\end{remark}

\begin{remark}
Factorization homology is proven to exist and be well-behaved in a $\otimes$-sifted cocomplete symmetric monoidal $\infty$-category \cite[Prop. 3.9]{Ayala_Francis_2019}. 
This is the case of $\Pr$ the bicategory of presentable cocomplete linear categories, but we do not know this result for $\Bimod$. 
Therefore a priori we should always write factorization homology with values in $\Pr$. 
It is a consequence of Theorem \ref{thm:main-result} that it indeed lies in $\Bimod$.
\end{remark}


\subsection{Decomposition properties}\label{sec:decomposition}

For the proof of Theorem \ref{thm:main-result}, we will want to decompose any morphism in $\SkCat_\II(\Sigma)$ into a composition of morphisms happening over disks. 
We need a strong version of this decomposition which does not pass to isotopy classes because we also need a decomposition result for relations between ribbon graphs.
By ribbon graph we always mean an admissible $\II$-colored ribbon graph, sometimes considered up to admissible skein relations.

\begin{definition}\label{def:DecompositionTimes}
Let $\overline{T}$ be an $\II$-colored ribbon graph from $(\iota, W)$ to $(\iota', W')$ in $\Sigma \times[0,1]$, not considered up to isotopy. 
A \textbf{decomposition} of $\overline{T}$ is the data
$$\DD= (t_i,\; \kappa^i:D_i\inj \Sigma,\; \rho_i:d_i\inj D_i)_{i=1,\dots,n}$$
of times $0=t_0<t_1<\cdots<t_n=1$ and embeddings of disks $\kappa^i:D_i\inj \Sigma$ and $\rho_i:d_i \inj D_i$ such that $\overline{T}\cap(\Sigma\times[t_{i-1},t_i])\subseteq \kappa^i(D_i) \times [t_{i-1},t_i]$ and $\overline{T}$ is transverse to each $\Sigma \times \{t_{i-1}\}$ and intersects it at the centers of the disks $\kappa^i\circ \rho_i (d_i)\subseteq \kappa^i(D_i)\cap \kappa^{i-1}(D_{i-1})$ with correct framing. 
For compatibility with the prescribed source and target, we ask that $\kappa^1\circ\rho_1 = \iota$ and add the convention that $\kappa^{n+1}\circ \rho_{n+1} = \iota'$. 
To satisfy the admissibility condition on $\II$-colored ribbon graphs, we also ask that each $\overline{T}\cap(\Sigma \times \{t_{i-1}\})$ is non-empty. 
See Figure \ref{fig:essentially-surjective-decomposition}.

From such a decomposition we obtain in particular that each $\overline{T}\cap(\Sigma\times[t_{i-1},t_i])\subseteq \kappa^i(D_i) \times [t_{i-1},t_i] \simeq D_i\times[0,1]$ defines a morphism $T_i: (\rho_i, W_i) \to (\rho_i', W_i')$ in $\SkCat_\II(D_i)$, where $\rho_i'$ is the unique embedding $\rho_i':d_{i+1} \inj D_i$ such that $\kappa^i\circ\rho_i' = \kappa^{i+1}\circ \rho_{i+1}$ and that the two colors $W_i$,$W_{i-1}'$ of $\overline{T}$ at the center of each $\kappa^i\circ \rho_i (d_i)$ agree. 

Each $\DD$ induces a decomposition \begin{equation}\label{eq:tangle-decomposition-over-disks}
    T= \kappa^n_*T_n \circ \cdots\circ \kappa^1_*T_1
\end{equation} where $T$ is the morphism in $\SkCat_\II(\Sigma)$ represented by $\overline{T}$.
\end{definition}

\begin{figure}
    \centering
    \begin{tikzpicture}[xscale = 1, yscale=0.7]

\begin{scope}[xshift=-1.5cm,yshift=4.5cm]
    \draw (1.5,0.1) .. controls (1.5,0.5) and (2.8,1)..(2.8,0.5).. controls (2.8,0.3) and (2.5,-0.3)..(2,-0.3)..controls (1.8,-0.3) and (1.5,-0.2) .. (1.5,0.1);
    \node at (2.1,0.2){$D_{i+1}$};
\draw (2.5,0.45) circle(0.16 and 0.19);
\draw (1.95,-0.13) circle(0.14);
\end{scope}

\begin{scope}[xshift=-0.8cm,yshift=1.5cm]
    \draw (2.2,-0.1)..controls(2.2,-0.6) and (0.8,-1)..(0.8,-0.6).. controls (0.8,-0.3) and (2.2,0.3).. (2.2,-0.1);
    \node at (1.65,-0.3){$D_i$};
\draw (1.95,-0.13) circle(0.14);
\draw (1.2,-0.5) circle(0.18 and 0.12);
\end{scope}

\begin{scope}[xshift=2cm,yshift=0cm]
    \fill[gray!40, opacity=0.5,draw=black,dashed] (2.2,-0.1)..controls(2.2,-0.6) and (0.8,-1)..(0.8,-0.6) -- ++(0,3) .. controls (0.8,2.7) and (2.2,3.3).. (2.2,2.9)-- ++(0,-3);  
    \draw (0,0) .. controls (0,-0.5) and (0.5,-1) .. (1,-1).. controls (1.5,-1) and (1.5,-0.7).. (2,-0.7) .. controls (2.5,-0.7) and (2.5,-1)..(3,-1)..controls (3.5,-1) and (4,-0.5)..(4,0).. controls (4,0.5) and (3.5, 1).. (3,1)..controls(2.5,1) and (2.5,0.7)..(2,0.7)..controls (1.5,0.7) and (1.5,1)..(1,1)..controls(0.5,1) and (0,0.5)..(0,0);
    \draw (0.6,0).. controls (0.8,-0.2) and (1.2,-0.2)..(1.4,0);
    \draw (0.75,-0.1).. controls (0.9,0) and (1.1,0)..(1.25,-0.1);
    \draw (2.6,0).. controls (2.8,-0.2) and (3.2,-0.2)..(3.4,0);
    \draw (2.75,-0.1).. controls (2.9,0) and (3.1,0)..(3.25,-0.1);
    \draw (2.2,-0.1)..controls(2.2,-0.6) and (0.8,-1)..(0.8,-0.6).. controls (0.8,-0.3) and (2.2,0.3).. (2.2,-0.1);  
\draw (1.2,-0.5) node[scale = 0.4]{$W_{i}$} circle(0.18 and 0.12); 
\end{scope}

\begin{scope}[xshift=2cm,yshift=3cm]
    \fill[gray!40, opacity=0.5,draw=black,dashed] (2.8,0.5).. controls (2.8,0.3) and (2.5,-0.3)..(2,-0.3)..controls (1.8,-0.3) and (1.5,-0.2) .. (1.5,0.1) -- ++(0,3) .. controls (1.5,3.5) and (2.8,4)..(2.8,3.5)-- ++(0,-3);  
    \node[draw, rectangle,inner sep = 2pt] (Ti+1) at (2.2,1.5){$T_{i+1}$};
    \draw (1.95,-0.13)..controls (1.95,0.5) and (2.2,1)..(Ti+1).. controls (2.2,2) and (2.5,3) .. (2.5,3.3);
    \begin{scope}[yshift=-3cm]
    \node[draw, rectangle,inner sep=2pt] (Ti) at (1.5,1.4){$T_{i}$};
    \draw (1.2,-0.4)..controls (1.2,0.2) and (1.5,0.7)..(Ti).. controls (1.5,2) and (1.95,2.2) .. (1.95,2.87);        
    \end{scope}
    \draw (0,0) .. controls (0,-0.5) and (0.5,-1) .. (1,-1).. controls (1.5,-1) and (1.5,-0.7).. (2,-0.7) .. controls (2.5,-0.7) and (2.5,-1)..(3,-1)..controls (3.5,-1) and (4,-0.5)..(4,0).. controls (4,0.5) and (3.5, 1).. (3,1)..controls(2.5,1) and (2.5,0.7)..(2,0.7)..controls (1.5,0.7) and (1.5,1)..(1,1)..controls(0.5,1) and (0,0.5)..(0,0);
    \draw (0.6,0).. controls (0.8,-0.2) and (1.2,-0.2)..(1.4,0);
    \draw (0.75,-0.1).. controls (0.9,0) and (1.1,0)..(1.25,-0.1);
    \draw (2.6,0).. controls (2.8,-0.2) and (3.2,-0.2)..(3.4,0);
    \draw (2.75,-0.1).. controls (2.9,0) and (3.1,0)..(3.25,-0.1);
    \draw (1.5,0.1) .. controls (1.5,0.5) and (2.8,1)..(2.8,0.5).. controls (2.8,0.3) and (2.5,-0.3)..(2,-0.3)..controls (1.8,-0.3) and (1.5,-0.2) .. (1.5,0.1); 
    \draw (2.2,-0.1)..controls(2.2,-0.6) and (0.8,-1)..(0.8,-0.6).. controls (0.8,-0.3) and (2.2,0.3).. (2.2,-0.1);
\draw (1.95,-0.13) circle(0.14);
\node[inner sep=0pt] (commoncircle) at (1.95,-0.13){};
\node[scale = 0.8, rectangle, draw] (condition) at (5, -0.5){$W_{i+1}=W_i'$};
\path[draw, ->] (condition) to (commoncircle);
\end{scope}

\begin{scope}[xshift=2cm,yshift=6cm]
    \draw (0,0) .. controls (0,-0.5) and (0.5,-1) .. (1,-1).. controls (1.5,-1) and (1.5,-0.7).. (2,-0.7) .. controls (2.5,-0.7) and (2.5,-1)..(3,-1)..controls (3.5,-1) and (4,-0.5)..(4,0).. controls (4,0.5) and (3.5, 1).. (3,1)..controls(2.5,1) and (2.5,0.7)..(2,0.7)..controls (1.5,0.7) and (1.5,1)..(1,1)..controls(0.5,1) and (0,0.5)..(0,0);
    \draw (0.6,0).. controls (0.8,-0.2) and (1.2,-0.2)..(1.4,0);
    \draw (0.75,-0.1).. controls (0.9,0) and (1.1,0)..(1.25,-0.1);
    \draw (2.6,0).. controls (2.8,-0.2) and (3.2,-0.2)..(3.4,0);
    \draw (2.75,-0.1).. controls (2.9,0) and (3.1,0)..(3.25,-0.1);
    \draw (1.5,0.1) .. controls (1.5,0.5) and (2.8,1)..(2.8,0.5).. controls (2.8,0.3) and (2.5,-0.3)..(2,-0.3)..controls (1.8,-0.3) and (1.5,-0.2) .. (1.5,0.1);
    \draw (2.5,0.45) node[scale = 0.4]{$W_{i+1}'$} circle(0.19 and 0.19);
\end{scope}

\draw (2,0) -- (2,6);
\draw (6,0) -- (6,6);
\node[xscale = 2, yscale = 1.3] at (1.9,1){$\hookrightarrow$};
\node[fill=white,inner sep=0pt] at (1.9,1.2){$\kappa^i$};
\node[xscale = 2, yscale = 1.3] at (1.9,4.5){$\hookrightarrow$};
\node[fill=white,inner sep=0pt] at (1.9,4.8){$\kappa^{i+1}$};

\begin{scope}[xshift=-3.5cm,yshift=4.5cm]
\draw (3.1,.8) node (rhoi1p) {} circle(0.16 and 0.19);
\node[above=1pt, scale = 0.8] at (rhoi1p) {$d_{i+1}'$};
\node[xshift=6mm,yshift=2mm] at (rhoi1p) {\tiny $\rho'_{i+1}$};
\node[xshift=6mm,xscale=1.4,yscale=0.8] at (rhoi1p) {$\hookrightarrow$};
\draw (2.55,0.17) node (rhoi1) {} circle(0.14);
\node[above=1pt, scale = 0.8] at (rhoi1) {$d_{i+1}$};
\node[xshift=6mm,yshift=1.5mm] at (rhoi1) {\tiny $\rho_{i+1}$};
\node[xshift=6mm,xscale=1.4,yscale=0.8] at (rhoi1) {$\hookrightarrow$};
\end{scope}

\begin{scope}[xshift=-2.8cm,yshift=1.5cm]
\draw (2.5,.1) node (rhoip) {}  circle(0.14);
\node[above=1pt, scale = 0.8] at (rhoip) {$d_{i}'$};
\node[xshift=5mm,yshift=1.5mm] at (rhoip) {\tiny $\rho'_i$};
\node[xshift=5mm,xscale=1.3,yscale=0.8] at (rhoip) {$\hookrightarrow$};

\draw (1.8,-0.5) node (rhoi) {} circle(0.18 and 0.12);
\node[above=1pt, scale = 0.8] at (rhoi) {$d_{i}$};
\node[xshift=5mm,yshift=1mm] at (rhoi) {\tiny $\rho_i$};
\node[xshift=5mm,xscale=1.3,yscale=0.8] at (rhoi) {$\hookrightarrow$};
\end{scope}


\draw[dashed] (6.1,0)-- ++(2,0) node[pos=1, right]{$t_{i-1}$};
\draw[dashed] (6.1,3)-- ++(2,0) node[pos=1, right]{$t_{i}$};
\draw[dashed] (6.1,6)-- ++(2,0) node[pos=1, right]{$t_{i+1}$};
\end{tikzpicture}
    \caption{Morphisms in the skein category decompose into ones living over overlapping disks.}\label{fig:essentially-surjective-decomposition}
\end{figure}

Not every $\overline{T}$ admits a decomposition, but each skein $T$ has a representative that can be decomposed as above.
The main requirement is that $\overline{T}$ is in good position:

\begin{definition}\label{def:good-position}
        A ribbon graph $\overline{T}$ is \textbf{in good position} if \begin{enumerate}
        \item the height function $\overline{T} \subseteq \Sigma\times [0,1]\to [0,1]$ has isolated critical points, and 
        \item at all times $t\in[0,1]$, there is an edge of $\overline{T}$ intersecting the level $\Sigma\times \{t\}$.
    \end{enumerate}
\end{definition}
We begin with a lemma that allows us to put ribbon graphs into good position. 

\begin{lemma}\label{lemma:morphisms-decompose-over-disks}
Every ribbon graph in good position admits a decomposition.
Every morphism $T$ in $\SkCat_\II(\Sigma)$ has a representative in good position, and is therefore a composition of morphisms happening over disks.
\end{lemma}
\begin{proof}
    Let $\overline{T}\subset \Sigma \times [0,1]$ be a ribbon graph in good position.
    We will build a decomposition of $\overline{T}$.
By definition $\overline T$ intersects every level $\Sigma\times\{t\}$ at a finite (non empty) collection of points, and this intersection is transverse for all but finitely many levels. 
For any $t$, let $D_t$ be a small neighborhood of the points $\overline T\cap(\Sigma\times\{t\})$ consisting of disks in $\Sigma\times\{t\}$.
    By continuity of the embedding, there exists an interval $U_t=(t-\varepsilon,t+\varepsilon)$ such that $\overline T\cap (\Sigma\times U_t)\subseteq D_t\times U_t$. 
    We obtain a cover $\{U_t\}_{t\in [0,1]}$ of the interval. 
By compactness we can find a finite sub-cover $(U_1,\dots,U_n)$. 
We restrict to a minimal subcover and order the intervals $U_i$ by increasing infimum. 
Each $U_i$ has associated disks $D_i$ and inclusion $\kappa^i:D_i\to \Sigma$.

Say $\oT$ has source and target $\iota$ and $\iota'$. We add to the list intervals $U_1 = [0,\varepsilon)$ and $U_n = (1-\varepsilon,1]$ for small enough $\varepsilon$ such that $\overline{T}\cap (\Sigma \times U_1)$ lies in $\iota(d)\times U_1$ and $\overline{T}\cap (\Sigma \times U_n)$ lies in $\iota'(d')\times U_n$. 

We can pick any $t_i \in U_i\cap U_{i+1}$ outside the finitely many critical values of the height function (and set $t_n=1$). 
We can moreover choose any embedding $\rho_i: d_i \to D_i$ corresponding to the inclusion of a small neighborhood of the finitely many framed points $\overline{T}\cap(\Sigma\times\{t_i\})$ which lies in both $D_i$ and $D_{i+1}$. 
Note that $\DD = (t_i, \kappa^i, \rho_i)_{i=1,\dots,n}$ is a decomposition of $\overline{T}$.    

Now consider any morphism $T$ in $\SkCat_\II(\Sigma)$. 
Having isolated critical points is a generic condition on the height function and we can find an isotopy representative $\oT$ that satisfies it. 
By admissibility, its source and target are non-empty and it satisfies the second condition of Definition \ref{def:good-position} in a small neighborhood of 0 and 1. 
Choose any point $p$ on any edge of $\oT$ and any path $\gamma$ in generic position going from $p$ to a point of height close to 1 and then to a point of height close to 0 in $(\Sigma \times [0,1])\smallsetminus \oT$. 
Isotope $\oT$ by pulling a small neighborhood of $p$ along the path $\gamma$. 
The resulting ribbon graph $\oT_{\gamma}$ is a representative of $T$ in good position. 
    \end{proof}
We want a similar result stating that the relations between ribbon graphs are local. 
We will need a strong form of locality that changes only one of the $T_i$'s in a decomposition at a time.
\begin{definition}
Two ribbon graphs $\oT$ and $\oT'$ are \textbf{locally skein equivalent} if they have a common decomposition $\DD = (t_i, \kappa^i:D_i\inj\Sigma, \rho_i)_{i=1,\dots, n}$ and there is an index ${i_0}$ such that $\oT$ and $\oT'$ agree strictly outside $\kappa^{i_0}(D_{i_0})\times [t_{{i_0}-1},t_{i_0}]$ and are related by isotopy and skein relations inside, i.e. $T_{i_0} = T_{i_0}'$ in $\SkCat_\II(D_{i_0})$.

\end{definition}
\begin{lemma}\label{lemma:relations-are-local}
The relations between morphisms in $\SkCat_\II(\Sigma)$ are generated locally. 
More precisely, two ribbon graphs $\oT$ and $\oT'$ in good position represent the same morphism in $\SkCat_\II(\Sigma)$ if and only if they are related by a sequence of local skein relations as described above (possibly changing the decomposition).
\end{lemma}
\begin{proof}
We show that two isotopic ribbon graphs are related by a sequence of local skein relations. 
Let $\varphi = (\varphi_s:\Sigma\times[0,1] \tilde\to \Sigma\times[0,1])_{s\in[0,1]}$ be an ambient isotopy of $\Sigma\times[0,1]$ with $\varphi_0=\id$ and $\varphi_1(\oT)=\oT'$. 
Denote $\oT_s := \varphi_s(\oT)$ for $s\in [0,1]$. 

We first show that we can choose $\varphi$ such that all ribbon graphs $\oT_s$ are in good position. 
For generic $\varphi$, each height function $\oT_s\to [0,1]$ has isolated critical points (we do not ask that these are non-degenerate.) 
Consider the operation $\oT\leftrightarrow \oT_{\gamma}$ of pulling a strand close to $\Sigma\times\{0\}$ and to $\Sigma\times\{1\}$ along a path $\gamma$ described in the proof of Lemma \ref{lemma:morphisms-decompose-over-disks}.
This operation is induced by an ambient isotopy which deforms $\oT$ through ribbon graphs in good position.
Applying the ambient isotopy $\varphi$ to $\oT_\gamma$, we obtain a family of ribbon graphs $\varphi_s(\oT_\gamma) = (\oT_s)_{\varphi_s(\gamma)}$ which are each obtained from the $\oT_s$ by pulling a strand along the path $\varphi_s(\gamma)$. 
The $\oT_s$ have non-empty boundary points, so we can find a global $\varepsilon$ such that the height function surjects onto $[0,\varepsilon)$ and $(1-\varepsilon,1]$. 
The ambient isotopy $\varphi$ is the identity on the boundary so there exists a small $\delta>0$ such that $\varphi(\Sigma\times[0,\delta))$ stays in $\Sigma\times[0,\varepsilon)$ and similarly near 1.
Choosing the path $\gamma$ that's $\delta$-close to $\Sigma\times\{0\}$ and $\Sigma\times\{1\}$, we get that every $\varphi_s(\oT_\gamma)$ is in good position. 
Finally, $\varphi_1(\oT_\gamma) = \oT'_{\varphi_1(\gamma)}$ is also isotopic to $\oT'$ through ribbon graphs in good position because $\oT'$ is. 
We have produced an isotopy $\oT\sim \oT_{\gamma}\sim\oT'_{\varphi_1(\gamma)}\sim \oT'$ that only passes through ribbon graphs in good position.

We can now suppose that $\varphi$ is such that every $\oT_s$ is in good position. 
For any fixed $s_0\in [0,1]$ we can choose a decomposition $(t_i, \kappa^i:D_i\inj \Sigma, \rho_i:d_i\inj D_i)_{i=1,\dots,n}$ of $\oT_{s_0}$. 
Now the condition that $(t_i, \kappa^i:D_i\inj \Sigma, \rho_i:d_i\inj D_i)_{i=1,\dots,n}$ is a decomposition of $\oT_s$ is open which means that \begin{enumerate}
    \item for small $\varepsilon$, $(t_i, \kappa^i:D_i\inj \Sigma, \rho_i:d_i\inj D_i)_{i=1,\dots,n}$ is a decomposition of $\oT_{s_0\pm\varepsilon}$, and
    \item for small $\delta_1,\dots,\delta_{n-1}$, $(t_i\pm \delta_i, \kappa^i:D_i\inj \Sigma, \rho_i:d_i\inj D_i)_{i=1,\dots,n}$ is a decomposition of $\oT_{s_0\pm \varepsilon}$
\end{enumerate}
Therefore we have a globally valid decomposition for the isotopy $\psi_{s_0} := (\varphi_s)_{s\in [s_0-\varepsilon,s_0+\varepsilon]}$. 
Consider the covering $(\Sigma \times (t_{i-1}-\delta_{i-1}, t_i+\delta_i))_{i=1,\dots,n}$ of $\Sigma\times [0,1]$. 
Up to higher isotopy, $\psi_{s_0}$ can be decomposed into a composition of isotopies supported on each $\Sigma \times (t_{i-1}-\delta_{i-1}, t_i+\delta_i)$, \cite[Corollary 1.3]{EK71}. 
Each of these isotopies is a local skein relation in the sense above, using a decomposition $(t_i\pm \delta_i, \kappa^i:D_i\inj \Sigma, \rho_i:d_i\inj D_i)_{i=1,\dots,n}$ with appropriate signs. 
Finally, by compactness of $[0,1]$, the isotopy $\varphi$ can be decomposed into a finite composition of such $\psi_{s_0}$'s, and we have shown locality for isotopies.

Finally, we can use ambient isotopies to confine any skein relation to a small ball in the thickened surface, hence to a local skein relation. 
\end{proof}

\subsection{Modified skein categories compute factorization homology}
Let $\II \subseteq\AA$ be a tensor ideal in a ribbon category. 
As a corollary of Theorem \ref{theorem:skcats-a-2-functor}, the restriction of $\SkCat_\II$ to $\Disk$ is a unital disk algebra $\Idiskalg$ in $\Bimod$.
\footnote{Passing to $\Pr$ gives us the disk algebra in corresponding to $\widehat\II$ under the equivalence between disk algebras and balanced braided categories.}
We will use the notation 
$$\Idiskalg=\SkCat_\II\vert_{\Disk}$$ 
and will denote its extension to the slice category $\Disk_{/\Sigma}$ by $\tilde{\Idiskalg}$. 
The central claim is:
\begin{theorem}\label{thm:main-result}
    There is an equivalence of 2-functors $$\int_-\Idiskalg \simeq \SkCat_\II(-)$$ between factorization homology of $\Idiskalg$ in $\Bimod$ and $\II$-skein categories.
\end{theorem}
Let $c : \tilde{\Idiskalg}\Rightarrow \const_{\SkCat(\Sigma)}$ be the strong natural transformation obtained by functoriality of $\SkCat$ and the embeddings into $\Sigma$ in the slice category, as described in Remark \ref{rmk:2funinducecocone}.
We will show that for any linear category $X$ the functor
\begin{equation}\label{eq:cstar-functor}
    c_* : \Hom_{\Cat}(\SkCat(\Sigma), X) \to \Nat(\tilde{\Idiskalg},\const_X).
\end{equation}
is an equivalence of categories. 
When $X = \Fun(\CC^{op},\Vect)$ is a presheaf category, we have $$\Hom_{\Bimod}(\SkCat(\Sigma),\CC):=\Hom_{\Cat}(\SkCat(\Sigma), X)$$ and the equivalence \eqref{eq:cstar-functor} exhibits $\SkCat_\II(\Sigma)$ as the 2-colimit defining factorization homology. 
We split the proof into three claims.

\paragraph{Claim: $c_*$ is essentially surjective.}
\begin{proof}
    We will show that each strong natural transformation $\alpha : \tilde{\Idiskalg} \Rightarrow \const_X$ is isomorphic to $c_*F$ for some functor $F: \SkCat_\II(\Sigma)\to X.$
    Remember that the data specifying $\alpha$ is a functor $\alpha_{\iota}:\SkCat(d)\to X$ for each embedding $\iota:d\hookrightarrow\Sigma$ and a natural isomorphism $\alpha_{\Phi}:\alpha_\iota \Rightarrow\; \alpha_{\iota'}\circ \SkCat(\rho)$ for each 1-morphism
    $\Phi= (\rho:d\to d',\ \varphi: \iota \Rightarrow \iota'\circ \rho)$ in the slice category $\Disk_{/\Sigma}$. 
    We will frequently use special 1-morphisms where $\iota=\iota'\circ\rho$ and 
    $\varphi$ is the identity isotopy, denoted $\Phi^\rho = (\rho, \id_\iota)$. 

    For objects, we set $$F(\iota, W) := \alpha_\iota(W)\ ,$$ where by abuse of notation $W$ is identified with the object $(\id_d,W) \in \SkCat(d)$. 
The main part of this proof is constructing $F$ on morphisms and showing that it is well defined. 

\noindent\textbf{Step 1: Definition:}

    Let $T:(\iota,W)\to(\iota',W')$ be a morphism in $\SkCat(\Sigma)$.
    By Lemma \ref{lemma:morphisms-decompose-over-disks}, we have a decomposition $$T= \kappa^n_*T_n \circ \cdots\circ \kappa^1_*T_1$$ for a finite collection of embeddings $\kappa^i:D_i\hookrightarrow \Sigma$. 
This decomposition $\DD$ is not unique.

    Applying $\alpha_{\kappa^i}:\SkCat(D_i)\to X$ on $T_i$, we get a morphism $\alpha_{\kappa^i}(T_i):\alpha_{\kappa^i}(\rho_i,W_i)\to \alpha_{\kappa^i}(\rho_i',W_i')$. 
We would like to compose these morphisms, but they are not composable on the nose. We have to conjugate each one by isomorphisms $(\alpha_{\Phi^{\rho_{i}}})_{W_{i}}: \alpha_{\kappa^i\rho_i}(W_i) \tilde\to \alpha_{\kappa^{i+1}}(\rho_{i},W_{i})$ and $(\alpha_{\Phi^{\rho_{i}'}})_{W_{i+1}}: \alpha_{\kappa^i\rho_i'}(W_i') \tilde\to\alpha_{\kappa^i}(\rho_i',W_i')$. 
    The composable versions of the $\alpha_{\kappa^i}(T_i)$ are \begin{equation}\label{eq:defbetai}
        \beta_i := (\alpha_{\Phi^{\rho_{i}'}})_{W_i'}^{-1} \circ \alpha_{\kappa^i}(T_i) \circ (\alpha_{\Phi^{\rho_{i}}})_{W_i}: \alpha_{\kappa^i\rho_i}(W_i)\to \alpha_{\kappa^i\rho_i'}(W_i')\ .
    \end{equation}
    We define \begin{equation}\label{eq:essentially-surjective-morphism-def}
        F(T,\DD):= \beta_n \circ \cdots \circ \beta_1 
    \end{equation} 
   and next prove that this does not depend on the choice of the decomposition $\DD$ of $T$.
   
    \vskip10pt
    \noindent\textbf{Step 2: Independence of decomposition for a fixed isotopy representative:}
    
We will show that for an isotopy representative $\overline{T}$ of $T$ in good position, the definition $$F(\overline{T}) := F(T,\DD)\quad \text{ for any decomposition $\DD$ of $\overline{T}$}$$ does not depend on $\DD$. 
Again, such a decomposition exists by Lemma \ref{lemma:morphisms-decompose-over-disks}.
 We prove independence using a common refinement of two decompositions of $\overline{T}$.
Consider two choices $$\DD = (t_i, \kappa^i: D_i\hookrightarrow \Sigma, \rho_i:d_i\hookrightarrow D_i)_{i\leq n}\text{ and }\DD' = (t_j', {\kappa^j}': D_j'\hookrightarrow \Sigma, \rho_j':d_j'\hookrightarrow D_j')_{j\leq n'}.$$
Up to using an intermediate choice with a small perturbation of the $t_i$'s, we can suppose that the $t_i$ and the $t_j'$ are all distinct. 
Therefore their union $(\hat t_l)_{l\leq n+n'-1}$ gives a vertically-finer decomposition of either. 
Each $\overline{T}\cap(\Sigma\times[\hat t_{l-1},\hat t_l])$ lies in both a $D_i$ and a $D_j'$ and therefore in their intersection.
A small neighborhood of the framed center of the image of each $\rho_i$, $\rho_i'$ also lies in their intersection. 
The intersection $D_i \cap D_i'$ may not be a disk, but we can re-decompose $\overline{T}$ there. We obtain a new decomposition $\widehat\DD$ of $\overline{T}$ which is both vertically and horizontally finer than $\DD$ and $\DD'$. 
It is related to either by a sequence of the following moves:
    \begin{enumerate}
    \item Object refinement: for $i>1$, replace $\rho_i:d_i \inj D_i$ by a smaller $\hat\rho_i: \hat d_i\to D_i$ such that $\hat\rho_i(\hat d_i)\subseteq \rho_i(d_i)$ which gives the same framed points on the centers.
    \item Horizontal refinement: suppose that there is a subdisk $\hat D_i \overset{\kappa}{\inj} D_i$ containing both $\rho_i(d_i)$ and $\rho_{i+1}(d_{i+1})$ and such that $\overline{T}\cap(\Sigma\times[t_{i}, t_{i+1}])$ lies in $\kappa^i(\hat D_i)\times[t_{i}, t_{i+1}]$. 
Replace $D_i$ by $\hat D_i$.
    \item Vertical refinement: suppose that $D_{i+1}=D_i$ at time $t_{i+1}$.
      Then forget $(t_{i+1}, D_{i+1}, \rho_{i+1})$. (The refinement is actually going in the other direction but this one is easier to write.) 
    \end{enumerate}
Let us show invariance of $F(\overline{T})$ under each of these moves: 
\begin{enumerate}
    \item For object refinement, we can factor $\hat\rho_i$ as $\hat d_i \overset{\kappa}{\inj} d_i \overset{\rho_i}{\inj} D_i$. 
Therefore by definition of a strong natural transformation \eqref{eq:strnatCompo} we have $\alpha_{\Phi^{\hat\rho_i}} = \alpha_{\Phi^{\rho_i}}\circ \alpha_{\Phi^{\kappa}}$. 
The extra term in $\beta_i$ cancels with its inverse in $\beta_{i-1}$ in the definition of $F(\overline{T})$.
    \item For horizontal refinement, we show that the two possible definitions of $\beta_i$ agree. We drop indices $i$ and restrict to $\Sigma\times[t_{i}, t_{i+1}]$ as the rest will not feature. We need to compare the top and bottom compositions below \begin{equation}
        \begin{tikzcd}
        \alpha_{\iota\kappa\rho}(W) \ar[d, equal] \ar[r, "(\alpha_{\Phi^{\kappa\rho}})_W"] & \alpha_\iota(\kappa\rho,W) \ar[r, "\alpha_\iota(\kappa_*T_1)"] \ar[d, <-, "(\alpha_{\Phi^{\kappa}})_{(\rho,W)}"] & \alpha_\iota(\kappa\rho',W') \ar[r, "(\alpha_{\Phi^{\kappa\rho'}})_W^{-1}"] \ar[d, <-, "(\alpha_{\Phi^{\kappa}})_{(\rho',W')}"] & \alpha_{\iota\kappa\rho'}(W') \ar[d, equal]
        \\
        \alpha_{\iota\kappa\rho}(W) \ar[r, "(\alpha_{\Phi^{\rho}})_W"]  & \alpha_{\iota\kappa}(\rho,W) \ar[r, "\alpha_{\iota\kappa}(T_1)"] & \alpha_{\iota\kappa}(\rho',W') \ar[r, "(\alpha_{\Phi^{\rho'}})_{W'}^{-1}"] & \alpha_{\iota\kappa\rho'}(W')
    \end{tikzcd}
    \end{equation}
    The first and last square commute by \eqref{eq:strnatCompo} and the middle square commutes by naturality of $\alpha_{\Phi^{\kappa}}$.
    \item For vertical refinement, we simply use that $\alpha_{\kappa^i}$ is a functor and preserves composition.
\end{enumerate}
We have shown that $F(\overline{T})$ is well defined.

\noindent\textbf{Step 3: Invariance under isotopy and skein relations:}
Let $\oT$ and $\oT'$ be two ribbon graphs in good position which represent the same skein $T$.
We must show that $F(\oT) = F(\oT')$.
We have shown Lemma \ref{lemma:relations-are-local} that it is enough to consider local skein relations $\overline{T}\sim \overline{T}'$ where $\overline{T}$ and $\overline{T}'$ have a common decomposition and agree strictly except that one of the $T_i$'s change by an equivalent skein in $\SkCat(D_i)$. 
The definition of $F$ is invariant under such a move, as $\alpha_{\kappa^i}(T_i)$ is unchanged.
 
    \noindent\textbf{Step 4: $c_*F$ is isomorphic to $\alpha$:}

    We define a modification $m:c_*F \Rrightarrow \alpha$. 
This is the data of a natural isomorphism $F \circ \SkCat(\iota) \Rightarrow \alpha_\iota$ between functors $\SkCat(D) \to X$ for every object $\iota:D\inj \Sigma$ of $\Disk_{/\Sigma}$.

    On objects of the form $(\id_D, W)$, $c_*F$ agrees with $\alpha$ by definition. 
On an object $(\rho:d\inj D,W)$, we have an isomorphism $(\alpha_{\Phi^\rho})_W:c_*F(\rho,W):= \alpha_{\iota\rho}(W) \to \alpha_\iota(\rho,W)$. 
We set $$m(\iota)_{(\rho,W)} := (\alpha_{\Phi^\rho})_W\ .$$ 
    
Let us check that this defines a natural transformation $m(\iota)$. 
Let $T: (\rho,W)\to (\rho',W')$ be a morphism in $\SkCat(D)$. 
Then $\SkCat(\iota)(T) = \iota_*T$ is already decomposed with $n=1$. 
Naturality is built in the definition \eqref{eq:defbetai}: $$F(\iota_*T) := \beta_1 := (\alpha_{\Phi^{\rho'}})_{W'}^{-1} \circ \alpha_{\iota}(T) \circ (\alpha_{\Phi^\rho})_{W}\ .$$

Now, let us check that $m$ is a modification. 
Let $\Phi = (\rho,\varphi):\iota \to \iota'$ be a 1-morphism in $\Disk_{/\Sigma}$. 
We need to check \eqref{diag:mutation-2-cell-compatibility} that 
\begin{equation}\label{eq:mIsaModif}
    (\alpha_{\Phi^\rho})_W \circ F(\SkCat(\varphi)_W) \overset?= (\alpha_\Phi)_W
\end{equation}
for any $(\id_d,W)$ in $\SkCat(d)$.

This is not immediate: in the left hand side $F(\SkCat(\varphi)_W)$ is computed using $\alpha$ on objects of $\Disk_{/\Sigma}$ evaluated on morphisms in the skein categories of disks. 
The right hand side uses $\alpha$ on 1-morphisms of $\Disk_{/\Sigma}$. 
We need to show that these are related, and essentially this is true because morphisms in the skein category encode isotopies happening over disks. 


We first show that \eqref{eq:mIsaModif} holds for the special case of an isotopy happening inside a bigger disk. 
Consider two embeddings $\rho, \rho':d\inj D$ and an isotopy $\psi:\rho\Rightarrow\rho'$ in $\Disk$. 
Then given an embedding $\kappa:D\inj\Sigma$, we have a 1-morphism $\Psi = (\rho',\kappa\psi: \kappa\rho\Rightarrow\kappa\rho')$ from $\kappa\rho:d\inj\Sigma$ to $\kappa:D\to \Sigma$ in $\Disk_{/\Sigma}$. 
This 1-morphism really comes from an isotopy living over a disk, and indeed it is isomorphic to the 1-morphism $\Phi^\rho=(\rho, \id_{\kappa\rho})$ by the following 2-morphism in $\Disk_{/\Sigma}$:
\begin{equation}
    \begin{tikzcd}[row sep = 40pt]
        d \ar[bend left,rr,"\rho",""{name=rho,below}] \ar[dr, "\kappa\rho"', ""{name=kaprho, above, pos = 0.8},""{name=kaprhoup, above, pos = 0.5}] & & D  \ar[dl, "\kappa", ""{name=kap, above}]\\
        & \Sigma 
        \ar[from = kaprho, to = 1-3, Rightarrow, bend right = 10pt, "\kappa\psi"{above= 1pt, pos = 0.25} , shorten > = 10pt]\ar[from = kaprhoup, to = 1-3, equal, dashed, bend left = 10pt, shorten > = 10pt]
    \ar[from = 1-1, to = 1-3, bend right,"\rho'" pos = 0.15,""{name=rhop,above}] 
        \ar[from = rho, to = rhop, Rightarrow, "\psi" pos = 0.4]
    \end{tikzcd}
\end{equation}
The compatibility \eqref{eq:strnat2morph} of $\alpha$ with 2-morphisms gives precisely
\begin{equation}\label{eq:mIsaModifLocally}
(\alpha_\Psi)_{W} = \alpha_\kappa(\SkCat(\psi)_W) \circ (\alpha_{\Phi^\rho})_W =: (\alpha_{\Phi^{\rho'}})_W \circ F(\SkCat(\psi)_W) \ .
\end{equation}

To prove the general formula we need to decompose $\Phi$ into such $\Psi$s and $\Phi^\rho$s. 
Let $\frac{1}{2}d$ denote the disjoint union of radius $\frac{1}{2}$ disks in the disjoint union of radius 1 disks $d$. 
Denote $\nu:\frac{1}{2} d\inj d$ the inclusion and $r=(\frac{2-s}{2}\id:d\to d)_{s\in[0,1]}$ the retraction of $d$ on the image of $\nu$. 
By continuity and compacity, there exists $t_0=0<t_1<\dots<t_n=1$ such that $\varphi_s(\frac12d) \subseteq \varphi_{t_i}(d)$ for $s\in [t_{i-1},t_{i}]$. 
In particular $\varphi\vert_{[t_{i-1},t_{i}]\times\frac12d}$ is of the form $(\varphi_{t_i})\psi_i$ for some isotopy $\psi_i$ between $\nu$ and another embedding $\rho_i':\frac{1}{2}d\inj d$. 
Together with the retraction, this gives a 1-morphism
\begin{equation}
\Psi_i = 
    \begin{tikzcd}[column sep = 100pt, row sep = 40pt]
        d \ar[d, "\varphi_{t_{i-1}}"  pos = 0.3] \ar[r, "\frac12\id"]& \frac12d  \ar[d, "\varphi_{t_{i-1}}\vert_{\frac12d}" pos = 0.3] \ar[r, "\nu"] & d  \ar[d, "\varphi_{t_i}"  pos = 0.3] \\
        \Sigma \ar[r, equal]&\Sigma \ar[r, equal]&\Sigma
        \arrow[from = 2-1, to = 1-2, "\varphi_{t_{i-1}}r"' sloped, Rightarrow, shorten = 10pt]
        \arrow[from = 2-2, to = 1-3, "\varphi\vert_{[t_{i-1},t_{i}]\times\frac12d}"' sloped, Rightarrow, shorten = 10pt]
    \end{tikzcd}
\end{equation}
induced by an isotopy $\psi_ir : \id_d \Rightarrow \rho_i'\circ\frac12\id_d$ in $\Disk$.

The 1-morphism $\Phi$ is isomorphic, using the retraction $r$, to the composition 
\begin{equation}
\Phi \simeq \Phi^{\rho}\circ \Psi_n\circ\cdots\circ\Psi_1
\end{equation}
Using that $\alpha$ preserves composition \eqref{eq:strnatCompo} and applying \eqref{eq:mIsaModifLocally} on every term gives us a formula for the right hand side of \eqref{eq:mIsaModif} which agrees with the formula for the left hand side given by the decomposition $$\DD = (t_i, \kappa^i:=\varphi_{t_i}:d\inj \Sigma, \rho_i:=\nu: \frac{1}{2}d \inj d)$$ of $\SkCat(\varphi)_W$.
\end{proof}

\paragraph{Claim: $c_*$ is faithful.}
\begin{proof}
    We will show that $c_*$ is injective on morphisms. 
This proof comes down to unraveling definitions until we can compare two collections of morphisms in $X$. 
Let $\eta : F \Rightarrow G$ be a natural transformation between arbitrary functors $F,G: \SkCat(\Sigma) \to X$. 
It is entirely determined by its components:
    \begin{equation}\label{eq:eta-components}
        \big\{ \eta_{(\iota, W)} \in \Hom_X\left(F(\iota,W), G(\iota,W)\right)\;\big|\; \iota : d \hookrightarrow \Sigma, W \in \SkCat(d) \big\}.
    \end{equation}
    The modification $c_*\eta : c_* F \Rrightarrow c_*G$ is likewise determined by its components: 
        \begin{equation}\label{eq:cstar-eta-components-1}
        \left\{\left. \begin{tikzcd}
            \SkCat(d) \ar[r,bend left=60,"c_*G(\iota)",""{name=G,below}] \ar[r,bend right=60,"c_*F(\iota)"',""{name=F,above}] &[5mm] X 
            \ar[from=F,to=G,Rightarrow,"c_*\eta_\iota" description]
        \end{tikzcd}\;\right|\; \iota : d \hookrightarrow \Sigma \right\}.
    \end{equation}
    Here we've used that $\tilde{\Idiskalg}(\iota:d \hookrightarrow \Sigma) = \SkCat(d)$ while $\const_X(\iota) = X$.
    Each  $c_*\eta_\iota$ is a natural transformation and so also determined by components. 
These are indexed over objects $W \in \SkCat(d)$.
    Hence \eqref{eq:cstar-eta-components-1} can be re-written as
    \begin{equation}\label{eq:cstar-eta-components-2}
        \left\{ \left. (c_*\eta_\iota)_W : c_*F(\iota)(W) \to c_*G(\iota)(W) \;\right|\; \iota : d \hookrightarrow \Sigma, W \in \SkCat(d)\right\}.
    \end{equation}
    Next, recall from the construction of $c_*$ in \ref{sec:factorization-homology} that $c_*F(\iota)(W) = F(\SkCat(\iota)(W)) =: F(\iota,W)$ (same for $G$) while $(c_*\eta_\iota)_W = \eta_{(\iota,W)}$.
    It follows that the collections in \eqref{eq:eta-components} and \eqref{eq:cstar-eta-components-2} have the exact same elements and indexing set.

    Therefore if $c_* \eta = c_* \eta'$ as modifications, it would mean that $\eta_{(\iota,W)} = \eta'_{(\iota,W)}$ as morphisms in $X$ for all objects $(\iota,W) \in \SkCat(\Sigma).$ This can happen if and only if $\eta = \eta'$ as natural transformations.
\end{proof}

\paragraph{Claim: $c_*$ is full.}
\begin{proof}
    We will show that $c_*$ is surjective on morphisms.

    Let $m : \alpha \Rrightarrow \beta$ be a modification between strong natural transformations $\alpha,\beta : \tilde{\Idiskalg} \Rightarrow \const_X$.
    We are interested in fullness, so we assume $\alpha = c_*F$ and $\beta = c_*G$ for functors $F,G: \SkCat(\Sigma) \to X$.
    As in \eqref{eq:cstar-eta-components-2}, $m$ is determined by its constituent natural transformations $m_\iota$ for $\iota : d \hookrightarrow \Sigma$, which in turn are described by their component morphisms $(m_\iota)_W$, with $W \in \SkCat(d)$.

    We define a (presumed) natural transformation $\eta : F \Rightarrow G$ by $\eta_{(\iota,W)} := (m_\iota)_W$. 
    It remains to show that $\eta$ is indeed natural.
    Let $T : (\iota: d \hookrightarrow \Sigma, W) \to (\iota' : d' \hookrightarrow \Sigma, W')$ be a morphism in $\SkCat(\Sigma)$.
    We will prove that the following diagram in $X$ is commutative:
    \begin{equation}\label{diag:fullness-naturality-diagram}
        \begin{tikzcd}
            F(\iota, W) \ar[r,"F(T)"] \ar[d,"\eta_{(\iota,W)}"'] & F(\iota', W') \ar[d,"\eta_{(\iota',W')}"] \\
            G(\iota, W) \ar[r,"G(T)"']\ar[ru,equal,"?"] & G(\iota',W').
        \end{tikzcd}
    \end{equation}

Let $T = \kappa^n_*T_n \circ \cdots \circ \kappa^1_*T_1$ be a decomposition over disks with the $T_i$ as in \eqref{eq:tangle-decomposition-over-disks}.
Each component $m_{\kappa^i}$ of the original modification is a natural transformation, so for each $T_i: (\rho_i, W_i) \to (\rho_i', W_i')$ in $\SkCat(d_i)$ we have a commutative square in the category $X$:
\begin{equation}
        \begin{tikzcd}
        F(\kappa^i\rho_i, W_i) \ar[r,"F(\kappa_*^iT_i)"] \ar[d,"\displaystyle\eta_{(\kappa^i\rho_i,W_i)}"'] & F(\kappa^i\rho_i', W_i') \ar[d,"\displaystyle\eta_{(\kappa^i\rho_i',W_i')}"] \\
        G(\kappa^i\rho_i, W_i) \ar[r,"G(\kappa_*^iT_i)"] & G(\kappa^i\rho_i',W_i').
    \end{tikzcd}
\end{equation}
 Gluing these diagrams together we have:
    \begin{equation*}
        \begin{tikzcd}[column sep = 14pt]
         F(\iota,W) \ar[r,equal]\ar[d,"\displaystyle\eta_{\iota,W}"] &[-6mm]  F(\kappa^1\rho_1,W_1) \ar[r,"F(\kappa_*^1T_1)"]\ar[d, "\displaystyle\eta_{\kappa^1\rho_1,W_1}"] &[3mm] F(\kappa^1\rho_1',W_1') \ar[r,equal]\ar[d,"\displaystyle\eta_{\kappa^1\rho_1',W_1'}"] &[-6mm] F(\kappa^2\rho_2,W_2) \ar[r,"F(\kappa_*^2T_2)"]\ar[d, "\displaystyle\eta_{\kappa^2\rho_2,W_2}"]&[3mm] \cdots \ar[r,"F(\kappa_*^{n}T_{n})"] &[3mm] F(\kappa^{n}\rho_n',W_{n}') \ar[r,equal]\ar[d,"\displaystyle\eta_{\kappa^{n}\rho_n',W_{n}'}"] &[-6mm] F(\iota',W') \ar[d, "\displaystyle\eta_{\iota',W'}"] \\
         G(\iota,W) \ar[r,equal] & G(\kappa^1\rho_1,W_1) \ar[r,"G(\kappa^1_* T_1)"] & G(\kappa^1\rho_1',W_1') \ar[r,equal] & G(\kappa^2\rho_2,W_2) \ar[r,"G(\kappa_*^2T_2)"]& \cdots \ar[r,"G(\kappa^{n}_*T_{n})"]& G(\kappa^{n}\rho_n',W_{n}') \ar[r,equal] & G(\iota',W') 
        \end{tikzcd}
    \end{equation*}
 It follows that \eqref{diag:fullness-naturality-diagram} is a commutative square, hence $\eta : \alpha \Rightarrow \beta$ is a natural transformation.
 Since $(c_*\eta_\iota)_W := \eta_{\iota,W} = (m_\iota)_W$, we conclude that $c_*$ is full.
\end{proof}
\begin{remark}
    When $\II$ contains the unit, i.e. $\II=\AA$, we obtain exactly \cite{CookeExcision}'s result. 
Note that our proof is quite different, as Cooke shows excision properties of skein categories and uses the characterization of factorization homology of \cite{Ayala_Francis_2019}. 
The original motivation for our approach was that it would be easier to generalize than the excision arguments.
\end{remark}

We will want to use the theorem above to relate our topological constructions to the results of \cite{BBJ}. 
Let us first translate our result to their context.
        
We can embed $\Bimod$ in the bicategory $\Pr$ of cocomplete presentable linear categories and cocontinuous functors by free cocompletion $\CC \mapsto \widehat\CC := \Fun(\CC^{op},\Vect)$. 
This embedding is symmetric monoidal. 
Its essential image is the full sub-bicategory spanned by categories with enough compact-projectives. 
See \cite{BCJReflDualPr, BJS} for more details. 

\begin{corollary}\label{cor:FHofE2inPr}
    Let $\EE$ be a $\Disk_2$-algebra in $\Pr$. 
Suppose that $\EE$ is cp-ribbon in the sense that it is cp-rigid and that its balancing is a ribbon structure on its subcategory of dualizable objects $\AA$. 
Then set $\II:=\EE^{cp}$ the ideal of compact-projective objects, we have an equivalence of symmetric monoidal 2-functors $$\displaystyle \int_-\EE \simeq \widehat{\SkCat_\II}(-)$$      between factorization homology with coefficients in $\EE$ and free cocompletions (or presheaf categories) of $\II$-skein categories.
\end{corollary}
\begin{proof}
By definition \cite[Definition-Proposition 1.3]{BJS}, $\EE$ being cp-rigid means that it has enough compact-projectives and its compact-projectives are dualizable. The first condition implies that $\EE \simeq \widehat\II$. 
The second asks that $\II\subseteq\AA$. 
It is a tensor ideal by the same proof as \cite[Proposition 4.2.12]{EGNO}. 
Note that $\AA$ is a priori a rigid balanced category, and the extra condition we ask is that the balancing on the dual is the dual of the balancing.

In the proof of the theorem above we actually show that \eqref{eq:cstar-functor} is an equivalence for any linear category $X$, not only for presheaf categories. 
By the universal property of free cocompletions we get $$\Hom_{\Pr}
(\widehat{\SkCat_\II}(\Sigma), X) \simeq \Hom_{\Cat}(\SkCat_\II(\Sigma),X) \simeq 
\Nat({\SkCat}\vert_{\Disk} \circ 
\operatorname{for},\const_X)$$ 
which, using the universal property of free cocompletion again, exhibits $\widehat{\SkCat_\II}(-)$ as the 2-colimit defining factorization homology in $\Pr$.

We just have to check that the $\Disk$-algebra $\widehat{\SkCat}\vert_{\Disk}$ is equivalent the one induced by $\EE$, namely that $\widehat{\SkCat}(\mathbb D)$ agree with $\EE$ as a balanced braided category. 
We already have an equivalence of categories $\EE \simeq \widehat \II \simeq \widehat{\SkCat}(\mathbb D)$. 
The braiding and balancing in $\widehat{\SkCat}(\mathbb D)$ between two objects of $\SkCat(\mathbb D) \simeq \II$ are defined, via the Reshetikhin--Turaev functor, by the braiding and balancing in $\II$ hence agree with the braiding and balancing of $\EE$ there. Now $\II$ generates either category under colimits, and natural transformations are determined by their value on such a generating subcategory.
\end{proof}
\begin{remark}\label{rmk:dep_on_II}
    Corollary \ref{cor:FHofE2inPr} implies that the $\II$-skein category of a surface does not depend on the ambient ribbon category $\AA$. The presence of $\AA$ is a nice way to package a list of conditions on $\II$, namely that its presheaf category $\widehat\II$ is a unital balanced braided monoidal category in $\Pr$ and is cp-ribbon. Importantly here, being unital is a property and not structure, as the unit is unique up to canonical isomorphism. 
\end{remark}
\begin{remark}\label{rmk:functoriality_in_II}
    We consider three senses in which $\SkCat_\II(\Sigma)$ is functoriality in $\II$.
    
    The most straightforward is with respect to ribbon functors $F:\AA \to \BB$ sending $\II\subseteq \AA$ to $\JJ\subseteq \BB$. 
    In this case we get a functor $\SkCat_\II(\Sigma)\to \SkCat_\JJ(\Sigma)$ sending an $\II$-colored ribbon graph to the same underlying graph with colors changed by $F$.

    Second, factorization homology is functorial in both its topological input and its algebraic input.
    Hence Theorem \ref{thm:main-result} implies that $\SkCat_{\II}\Sigma$ is functorial with respect to balanced braided monoidal cocontinuous functors $\widehat\II \to \widehat\JJ$.
    In this case, we get a cocontinuous functor $\widehat\SkCat_\II(\Sigma)\to \widehat\SkCat_\JJ(\Sigma)$.
    This should admit a skein theoretic description similar to the description of the distinguished presheaf in Remark \ref{rmk:fake-empty}.

    Finally, given two $\Disk_2$-algebras $\widehat \II$ and $\widehat \JJ$, there is a notion of monoidal bimodules between them, discussed as central algebras in \cite{BJS} or as constructible factorization algebras with respect to the stratification $\Rr \subseteq \Rr^2$ in \cite{Scheimbauer}.
    In good cases this data determines a defect between $\II$- and $\JJ$-skein theory, a skein-theoretic description of which will appear in future work of the first author and D. Jordan.
\end{remark}

\subsection[Computation of I-skein algebras]{Computation of  $\II$-skein algebras}\label{sec:computionation-of-skalg}
Showing that $\II$-skein categories agree with factorization homology gives us access to a new toolkit.
The following statement follows directly from \cite{BBJ}, where they introduce the moduli algebra $A_\Sigma$ and show it is isomorphic as an algebra object in $\EE$ to $\mathcal{L}^{\otimes 2g+n-1}$ with a product twisted by appropriate braidings in \cite[Theorem 5.14]{BBJ}. Here $\LL$ is Lyubashenko's coend $\mathcal{L} := \int^{X \in \II} X \otimes X^* \in \EE$ \cite{LyubashenkoModtrasnfoTensorCats}.
\begin{corollary}\label{cor:ModSkAlgIsInvCoend}
Let $\EE,\AA$ and $\II$ be in Corollary \ref{cor:FHofE2inPr} and let $\Sigma$ be a connected genus $g$ surface with $n\geq 1$ punctures. 
There is an isomorphism of algebras
$$\SkAlg_\II(\Sigma) \simeq \Hom_\EE(\idty, \mathcal{L}^{\otimes 2g+n-1})$$ where $\mathcal{L}^{\otimes 2g+n-1}$ is endowed with a product twisted by appropriate braidings \cite[Definition 5.11]{BBJ}.
\end{corollary}
\begin{proof}
The $\II$-skein algebra is exactly the algebra of invariants $\Hom_{\widehat\II}(\idty, A_\Sigma)$ of the moduli algebra $A_\Sigma$ studied in \cite{BBJ}. 
Indeed $A_\Sigma$ is defined to be the internal endomorphism algebra $\underline{\End}_{\int_\Sigma\widehat\II}(\Dist_\Sigma)$ whose defining universal property asks that $$\Hom_{\widehat\II}(\idty, A_\Sigma) \simeq \Hom_{\int_\Sigma\widehat\II}(\idty\rhd\Dist_\Sigma, \Dist_\Sigma)\ .$$
Here $\rhd$ is the action induced by inserting a disk though the boundary of $\Sigma$. 

By Theorem \ref{thm:main-result}, $$\Hom_{\int_\Sigma\widehat\II}(\idty\rhd\Dist_\Sigma, \Dist_\Sigma)\simeq\Hom_{\widehat{\SkCat}_\II(\Sigma)}(\Dist_\Sigma, \Dist_\Sigma) =:\SkAlg_\II(\Sigma)\ .$$

\end{proof}
The moduli algebra and its algebra of invariants have been well-studied in the literature. The best known case \cite{BFRqmoduliIII} is when $\AA = H\text{--}\mathrm{mod}^{fd}$ for a non-semisimple finite-dimensional 
Hopf algebra $H$ (e.g. small quantum groups at roots of unity) and $\II$ the tensor ideal of projectives, so $\widehat\II \simeq H\text{--}\mathrm{mod}$. 
For $H$ the small quantum group associated with $\mathfrak{sl}_2$ at a $p$-th root of unity and $\Sigma = S^1\times [0,1]$ the annulus, it is shown to be $3p-1$ dimensional in \cite{GainutdinovTipuninInvCoend} and \cite[Sections 2.2 and 3]{FaitgSLFsl2} using an explicit basis.
This explicit description can be used to show non-surjectivity of the canonical map $\SkAlg_\AA(\Sigma) \to \SkAlg_\II(\Sigma)$.
By arguments of Matthieu Faitg, its image is $2p$-dimensional for the annulus.

\appendix

\section{Strong natural transformations and modifications}\label{sec:modifications}
We describe the category $\Nat(F,G)$ which appears in Definition \ref{def:factorization-homology} of factorization homology.
In what follows we assume that $\CC$ and $\DD$ are 2-categories and that $F,G:\CC\to\DD$ are 2-functors. We assume no strictness and give the pseudo version of every notion, but we will suppress unitors and associators, which can easily be inserted.

\begin{definition}[Definition 4.2.1 of \cite{JohnsonYauBook}]
    A \textbf{strong natural transformation} $\alpha: F\Rightarrow G$ between two 2-functors $F,G:\CC\to\DD$,
        \begin{equation}
        \begin{tikzcd}
        \mathcal{C}  \ar[bend left,rr,"G",""{name=G,below}] \ar[bend right,"F"',rr,""{name=F,above}] & \ & \mathcal{D},
        \ar[from=F,to=G,Rightarrow,"\alpha"] 
    \end{tikzcd}
\end{equation}
is the data of a 1-morphism $\alpha_A: F(A) \to G(A)$ in $\DD$ for each object $A\in\CC$ and an invertible 2-morphism $\alpha_f$ in $\DD$ filling the square
\begin{equation}
    \begin{tikzcd}
        F(A_1) \arrow[d, "F(f)"'] \arrow[r, "\alpha_{A_1}"] & G(A_1) \arrow[d, "G(f)"]\\
        F(A_2)  \arrow[r, "\alpha_{A_2}"] & G(A_2) \arrow[to = 1-2, from= 2-1, Leftrightarrow, "\alpha_f"]
    \end{tikzcd}
\end{equation}
for each 1-morphism $f:A_1 \to A_2$ in $\CC$. 
These should be natural in the following sense.
We require that $\alpha_{{\id}_A}=\id_{\alpha_A}$ for each object $A$ and that for composable $A_1\overset{f}{\to} A_2 \overset{g}{\to} A_3$ we have an equality of 2-morphisms in $\DD$:
    
    \begin{equation}\label{eq:strnatCompo} \begin{tikzcd}
        F(A_1) \arrow[d, "F(g\circ f)"'] \arrow[r, "\alpha_{A_1}"] & G(A_1) \arrow[d, "G(g\circ f)"]\\
        F(A_3)  \arrow[r, "\alpha_{A_3}"] & G(A_3) \arrow[to = 1-2, from= 2-1, Leftrightarrow, "\alpha_{g\circ f}"]
    \end{tikzcd} \quad=\quad \begin{tikzcd}
        F(A_1) \arrow[d, "F(f)"'] \arrow[r, "\alpha_{A_1}"] & G(A_1) \arrow[d, "G(f)"]\\
        F(A_2) \arrow[d, "F(g)"'] \arrow[r, "\alpha_{A_2}"] & G(A_2) \arrow[d, "G(g)"]\arrow[to = 1-2, from= 2-1, Leftrightarrow, "\alpha_f"] \\
        F(A_3)  \arrow[r, "\alpha_{A_3}"] & G(A_3) \arrow[to = 2-2, from= 3-1, Leftrightarrow, "\alpha_g"]
    \end{tikzcd}    
    \end{equation}
    Finally, for every 2-morphism $h: f \Rightarrow g$ in $\CC$ we must have the following equality: 
    \begin{equation}\label{eq:strnat2morph}\begin{tikzcd}
        F(A_1) \arrow[d, "F(f)"'] \arrow[r, "\alpha_{A_1}"] & G(A_1) \arrow[d, "G(f)"'] \arrow[r, equal] & G(A_1) \arrow[d, "G(g)"]\\
        F(A_2)  \arrow[r, "\alpha_{A_2}"'] & G(A_2) \arrow[to = 1-2, from= 2-1,shift left,Leftrightarrow,,"\alpha_f"] \arrow[r, equal] & G(A_2) \arrow[to = 1-3, from= 2-2, Rightarrow, "G(h)"]
    \end{tikzcd}
    \quad=\quad
    \begin{tikzcd}
        F(A_1) \arrow[d, "F(f)"'] \arrow[r, equal]& F(A_1) \arrow[d, "F(g)"'] \arrow[r, "\alpha_{A_1}"] & G(A_1) \arrow[d, "G(g)"]\\
        F(A_2) \arrow[r, equal] & F(A_2)  \arrow[r, "\alpha_{A_2}"'] & G(A_2) \arrow[to = 1-2, from= 2-1, shift left, Rightarrow, "F(h)"] \arrow[to = 1-3, from= 2-2, Leftrightarrow, "\alpha_g"]
    \end{tikzcd}
    \end{equation}
\end{definition}

\begin{definition}[Definition 4.4.1 of \cite{JohnsonYauBook}]
    A \textbf{modification} $m : \alpha \Rrightarrow \beta$ between two strong natural transformations $\alpha,\beta : F \Rightarrow G$
    is the data of a 2-morphism $m_A : \alpha_A \Rightarrow \beta_A$ 
    for every object $A \in \mathcal{C}$.
    They must be compatible with the 2-cell components of $\alpha$ and $\beta$, namely for every 1-morphism $f : A_1 \to A_2$ we have an equality of 2-morphisms in $\DD$: 
\begin{equation} \label{diag:mutation-2-cell-compatibility}
\begin{tikzcd}
        F(A_1) \arrow[d, "F(f)"'] \arrow[r, "\alpha_{A_1}"] & G(A_1) \arrow[d, "G(f)"]\\
        F(A_2)  \arrow[r, "\alpha_{A_2}",""{name=alpha,below}] \arrow[d, equal] & G(A_2) \arrow[to = 1-2, from= 2-1, Leftrightarrow, "\alpha_f"] \arrow[d, equal] \\
        F(A_2) \arrow[r, "\beta_{A_2}"', ""{name=beta,above}]& G(A_2)
        \arrow[to = 1-2, from= 2-1, Leftrightarrow, "\alpha_f"] \ar[from=beta,to=alpha,Leftarrow,"m_{A_2}"]
    \end{tikzcd} \quad=\quad
    \begin{tikzcd}
    F(A_1) \arrow[d, equal] \arrow[r, "\alpha_{A_1}",""{name=alpha,below}] & G(A_1) \arrow[d, equal] \\
        F(A_1) \arrow[d, "F(f)"'] \arrow[r, "\beta_{A_1}", ""{name=beta,above}] & G(A_1) \arrow[d, "G(f)"]\\
        F(A_2)  \arrow[r, "\beta_{A_2}"] & G(A_2) \arrow[to = 2-2, from= 3-1, Leftrightarrow, "\beta_f"] \ar[from=beta,to=alpha,Leftarrow,"m_{A_1}", shorten=5]
    \end{tikzcd} 
\end{equation}
\end{definition}
\begin{definition}[Corollary 4.4.13 of \cite{JohnsonYauBook}]
    Given 2-functors $F,G:\CC\to\DD$ we write $\Nat(F,G)$ for the category of strong natural transformations from $F$ to $G$ and modifications between these. 
\end{definition}
 Note that in this paper we work with (2,1)-categories. 
This implies in particular that every modification is invertible, and that the category $\Nat(F,G)$ is a groupoid.

\small
\bibliography{main.bib}
\bibliographystyle{alpha}
\end{document}